\documentclass[review, english]{elsarticle}

\journal{Stochastic Processes and their Applications}

\usepackage[utf8]{inputenc}
\usepackage[T1]{fontenc}
\usepackage{lmodern}
\usepackage{babel}

\usepackage{mathrsfs}
\usepackage{amsmath}
\usepackage{amssymb} % loads package amsfont by default.
\usepackage{amsthm}

\usepackage{graphicx}
\usepackage[font = footnotesize, skip = 2pt]{caption}

\usepackage{booktabs}

\usepackage[unicode=true, colorlinks = true, linkcolor=blue, citecolor=blue]{hyperref}

\usepackage{mleftright} % corrects spacing after \left( and before \right), etc
\mleftright

\usepackage[section]{placeins} % prevents floats from violating a `section barrier'

\usepackage{accents} % For defining an underbar which displays better than underline

\usepackage[inline]{enumitem}

\usepackage{appendix}

\usepackage[dvipsnames]{xcolor}

\newtheorem{theorem}{Theorem}
\newtheorem{corollary}{Corollary}
\newtheorem{lemma}{Lemma}
\newtheorem{proposition}{Proposition}

\theoremstyle{remark}
  \newtheorem*{remark}{Remark}
\theoremstyle{definition}
\newcommand{\E}{\mathbb{E}}
\newcommand{\Prob}{\mathbb{P}}
\newcommand{\floor}[1]{\lfloor{#1}\rfloor}
\newcommand{\norm}[1]{\left\Vert{#1}\right\Vert}
\newcommand{\abs}[1]{\left\vert{#1}\right\vert}

\DeclareMathOperator*{\supp}{supp}
\DeclareMathOperator*{\Var}{Var}
\DeclareMathOperator*{\Cov}{Cov}

\DeclareMathOperator*{\Null}{Null}
\DeclareMathOperator*{\rank}{rank}
\DeclareMathOperator*{\Ran}{Ran}

\begin{document}

\begin{frontmatter}

\title{\textsc{Conjugate Processes}: \\ Theory and Application to Risk Forecasting}

\author[mymainaddress]{Eduardo Horta\corref{mycorrespondingauthor}}
\cortext[mycorrespondingauthor]{Corresponding author}
\ead{eduardo.horta@ufrgs.br}

\author[mymainaddress]{Flavio Ziegelmann}

\address[mymainaddress]{Universidade Federal do Rio Grande do Sul, Department of Statistics, 9500 Bento Gonçalves Av., 43--111, Porto Alegre, RS, Brazil 91509-900.}

\begin{abstract}
Many dynamical phenomena display a cyclic behavior, in the sense that time can be partitioned into units within which distributional aspects of a process are homogeneous. In this paper, we introduce a class of models - called conjugate processes - allowing the sequence of marginal distributions of a cyclic, continuous-time process to evolve stochastically in time. The connection between the two processes is given by a fundamental compatibility equation. Key results include Laws of Large Numbers in the presented framework. We provide a constructive example which illustrates the theory, and give a statistical implementation to risk forecasting in financial data.
\end{abstract}


\end{frontmatter}

%\linenumbers

\section{Introduction}
Many dynamical phenomena display a cyclic behavior, in the sense that time can be partitioned into units within which certain distributional aspects of a process are homogeneous. This idea is the starting point of the theory developed in Bosq~\cite{bosq2000linear}, for instance. The standard probabilistic approach to modeling the evolution of a system over time usually begins with specification of a certain probability measure on the space of sample paths, induced by a family of finite--dimensional distributions. In this setting, consideration of conditional probabilities usually involves the notion of `past information' as summarized by a filtering or the past trajectory of the process. We shall take a different approach, by introducing a latent process which permits us to interpret the cyclic character of a process in a conditional, distributional sense. We consider the following model. A sequence of random probability measures $\xi_0,\xi_1,\dots,\xi_t,\dots$ evolves stochastically in time. Associated to these probabilities is a continuous time, real-valued stochastic process $\left(X_\tau:\,\tau\geq 0\right)$ that satisfies the following condition, for each Borel set $B$ in the real line,
\begin{equation}\label{eq:weakly-conjugate-model}
	\Prob\left[X_\tau\in B\,\vert\,\xi_0, \xi_1, \dots\right] = \xi_t\left(B\right),\quad \tau\in\left[t,t+1\right).
\end{equation}
We shall call each interval $\left[t,t+1\right)$ the $t$-th \emph{cycle}. Of course, equation~\eqref{eq:weakly-conjugate-model} implies that, for $\tau\in\left[t,t+1\right)$, one has $\Prob\left[X_\tau \in B\,\vert \,\xi_0,\dots, \xi_t\right] = \Prob\left[X_\tau \in B\,\vert\,\xi_t\right] = \xi_t\left(B\right)$. This can be interpreted as meaning that the process $\left(X_\tau\right)$ has marginal conditional distribution $\xi_t$ during cycle $t$, and that past and future information about the $\xi_j's$ is to some extent irrelevant when $\xi_t$ is given. Little further probabilistic structure is imposed on $\left(X_\tau\right)$. Notice however that the distribution of $\left(X_\tau:\,\tau\geq 0\right)$ is not entirely determined by \eqref{eq:weakly-conjugate-model}. The model is potentially useful in situations where there is a natural notion of a cycle in the behavior of the process $\left(X_\tau\right)$, and where the main interest concerns statistical (i.e.\ distributional) aspects of the process, rather than `sample-path' aspects, within each cycle. Possible applications include temperature measurements and intraday stock market return processes, the latter of which we illustrate below with a real data set. This model does have a Bayesian flavor, in that the distribution of the random variables $X_\tau$ are themselves random elements in a space of  probability measures, but we shall not sail in this direction here. From now on we will go without saying that the index sets for $t$ and $\tau$ are $0,1,2,\dots$ and $\mathbb{R}^+$ respectively. A pair $\left(\xi_t,\,X_\tau\right)$, where $\left(\xi_t\right)$ is a sequence of random probability measures, and $\left(X_\tau\right)$ is a process satisfying the compatibility condition \eqref{eq:weakly-conjugate-model}, will be called a \emph{conjugate process}. $\left(\xi_t\right)$ is the \emph{latent} (or \emph{hidden}) \emph{distribution process} and $\left(X_\tau\right)$ is the \emph{observable process}. Notice that the probabilistic structure of the latent process $\left(\xi_t\right)$ can be defined `prior' to even mentioning the observable process $\left(X_\tau\right)$. In particular, the latter can be essentially anything as long as \eqref{eq:weakly-conjugate-model} holds.

In our model the evolution of $\left(X_\tau\right)$ over time is driven by the measure valued process $\left(\xi_t\right)$. Consideration of random measures, together with the Hilbert space embedding introduced in Section~\ref{sec:Hilbert-embedding}, places our methodology in the realm of probability in function spaces, which has a long--standing tradition in the probability literature. The theory of probability measures on function spaces first rose from the need to interpret stochastic processes as random elements with values in spaces of functions, the original insight likely due to Wiener, who constructed a probability measure on the space of continuous functions -- namely, Brownian motion -- yet before Kolmogorov's axiomatization of probability theory. It eventually became clear that a convenient and quite general approach is to consider probability measures in metric spaces, as established for instance in the classic texts Billingsley~\cite{billingsley2009convergence} and Parthasarathy~\cite{parthasarathy1967probability}. See also Van der Vaart and Wellner~\cite{van1996weak} for a modern account. A derived literature considers random elements (and hence probability measures) in Banach spaces, of which a very good exposition can be found in the classic texts Ledoux and Talagrand~\cite{ledoux1991probability} and Vakhania et al.~\cite{vakhania1987probability}. For stationary sequences and linear processes in Banach spaces, the monograph from Bosq~\cite{bosq2000linear} is a complete account.

The concept of random probability measures is very important in the theory of Bayesian nonparametrics -- see Ghosh and Ramamoorthi~\cite{ghosh2003bayesian}. Our approach however places us closer to the theory of inference on objects pertaining to function spaces, which in the statistics literature has come to be known as Functional Data Analysis (hereafter \textsc{fda}) -- see the cornerstone monograph by Ramsay and Silverman~\cite{ramsay1998functional} for a thorough treatment on the topic. In recent years, \textsc{fda} has received growing attention from researchers of a wide spectrum of academic disciplines; see for instance the collection edited by Dabo-Niang and Ferraty~\cite{dabo2008functional} for a discussion on recent developments and many applications, and also Benko et al.~\cite{benko2009common} who provide a very interesting application of \textsc{fda} to estimation of implied volatility. A blend of theory and application can be found in Ferraty and Vieu~\cite{ferraty2006nonparametric} and Damon and Guillas~\cite{damon2005estimation}. A central technique in this context is that of functional principal components analysis. At short, such methodology -- whose foundation lies in the Karhunen-Loève Theorem -- seeks a decomposition of the observed functions as orthogonal projections onto a suitable orthonormal basis corresponding to the eigenfunctions of a covariance operator. See Panaretos and Tavakoli~\cite{panaretos2013cramer} for a motivation. The spectral representations in Theorems~\ref{thm:spectral-representation-F_t} and \ref{thm:spectral-representation-F_t-Rmu} are straightly related to this methodology. Hall and Vial~\cite{hall2006assessing} study functional data in the presence of imprecise measurement -- due to rounding, experimental errors, etc. -- a scenario in which some complications arise regarding estimation of the covariance operator. Bathia et al.~\cite{bathia2010identifying} tackle this issue in a functional time series framework. Our original insight, which ultimately led us to proposing equation \eqref{eq:weakly-conjugate-model}, was inspired by an application of their methodology to modeling the dynamics of probability density functions. In our framework the observable random functions under consideration are the empirical {cumulative distribution functions (hereafter \textsc{cdf}s)} of the process $\left(X_\tau\right)$ sampled in each cycle. We also consider estimation of functional parameters, namely the eigenfunctions introduced in equation~\eqref{eq:F_t-spectral-representation}.

The text is organized as follows. In the next section, a formalism is proposed for the ideas presented above; some notation is established, and basic properties of conjugate processes are derived under this formalism. The most important result here is Theorem~\ref{thm:spectral-representation-F_t}, which establishes a spectral representation, in a suitable Hilbert space, of the latent process $\left(\xi_t\right)$. In Section~\ref{sec:inference} we lean towards inference, the ultimate goal being estimation -- through sampling the observable process $\left(X_\tau\right)$ only -- of quantities related to an alternate spectral representation of the latent process. The latter representation -- which is the content of Theorem~\ref{thm:spectral-representation-F_t-Rmu} -- is introduced inspired by the methodology put forth by Bathia et al.~\cite{bathia2010identifying} which solved an identification problem in the functional time--series framework. Consistency of proposed estimators is established in Theorems~\ref{thm:LLN-for-Fhat}, \ref{thm:LLN-for-psihat-no-rate} and \ref{thm:LLN-for-psihat}, whereas Proposition~\ref{thm:LLN_etahat} and Corollary~\ref{thm:LLN2_etahat} establish a ``quasi--consistency'' property for an estimator of a latent $\ell^2$ time series that characterizes the dynamic aspects of $\left(\xi_t\right)$. In Section~\ref{sec:example} a constructive, if simple, example is given which elucidates the theory. Finally, an application to high--frequency financial data is given in Section~\ref{sec:application}. Assuming high--frequency financial returns share the same marginal distribution in each day, but allowing said marginals to vary stochastically from day to day, we are able to characterize the dynamic aspects of the latent distribution process via a scalar time series. The latter can be modeled in a standard fashion -- say, as an \textsc{arma} process -- and forecasts can be generated from which measures of risk may be recovered and predicted. In summary, the method allows one, at the end of each day, to use current information to forecast \emph{distributional} aspects of the observable process $\left(X_\tau\right)$ in the next cycle. In Appendix~\ref{sec:bit-theory} we provide some of the necessary theoretical background. The proofs of all propositions in the main text are relegated to Appendix~\ref{sec:proofs}. In Appendix~\ref{sec:estimation} computational aspects of the methodology are discussed, and explicit formulas of the proposed estimators are given.

\section{Formalism and basic properties}
We consider given a probability space $\left(\Omega,\mathscr{A},\Prob\right)$ where a conjugate process $\left(\xi_t, X_\tau\right)$ is defined. In the above discussion we have introduced $\xi_0, \xi_1,\dots$ as a sequence of random probability measures which evolves stochastically in time. Formally, the meaning of this assertion is embodied in the following assumption, which we make throughout this paper.

\begin{description}[leftmargin=!]
	\item[Assumption~S] \emph{$\left(\xi_t:\, t=0,1,\dots\right)$ is a stationary sequence of random elements taking values in the space $M_1\left(\mathbb{R}\right)$}.
\end{description}

Here, $M_1\left(\mathbb{R}\right)$ is the set of all Borel probability measures on $\mathbb{R}$. We refer the reader to the Appendix \ref{sec:bit-theory} for further definitions and some basic theory. Under Assumption~S, the expectation (technically, the baricenter) of $\xi_t$ does not depend on $t$, and we shall denote it by $\E\xi_0$. Regarding the unconditional marginals of $\left(X_\tau\right)$ we have the following result.

\begin{lemma}\label{thm:X_tau-unconditional-distribution}
The random variable $X_\tau$ has marginal distribution $\E\xi_0$.
\end{lemma}

In what follows, it will be convenient to define the $\sigma$--field $\Xi := \sigma\left(\xi_0,\xi_1,\dots\right)$. Let us say that a conjugate process $\left(\xi_t,\,X_\tau\right)$ is \emph{cyclic--independent} if, conditional on $\Xi$, the stochastic processes $\left(X_\tau:\,\tau\in\left[0,1\right)\right)$, $\left(X_\tau:\,\tau\in\left[1,2\right)\right)$, $\dots$, $\left(X_\tau:\,\tau\in\left[t,t+1\right)\right)$, $\dots$ form and independent sequence\footnote{The concept is easily understood, but expressing it in terms of the finite dimensional distributions is a tedious task.}. In particular, cyclic independence implies that the random variables $X_{\tau_1}$ and $X_{\tau_2}$ are conditionally independent whenever $\floor{\tau_1}\neq\floor{\tau_2}$, that is,
\begin{equation}\label{eq:cyclic-independence}
\Prob\left(X_{\tau_1}\in B_1,\,X_{\tau_2}\in B_2\,\vert\,\Xi\right) = \xi_{t_1}\left(B_1\right)\xi_{t_2}\left(B_2\right)
\end{equation}
for Borel sets $B_j$ whenever $\tau_j\in\left[t_j,\,t_j+1\right)$ and $t_1\neq t_2$.

The concept of cyclic independence may appear restrictive at a first glance; for instance, one could assume it implies that the random variables $X_{\tau_1}$ and $X_{\tau_2}$ are also \emph{unconditionally} independent when $\floor{\tau_1}\neq\floor{\tau_2}$. Fortunately, this is not the case: as is easily seen from the above identity, the random variables $X_{\tau_1}$ and $X_{\tau_2}$ will be unconditionally independent if and only if $\E \xi_{t_1}\left(B_1\right)\xi_{t_2}\left(B_2\right) = \E \xi_{t_1}\left(B_1\right)\E\xi_{t_2}\left(B_2\right)$ for all Borel sets $B_1, B_2$.

A first interesting property of cyclic--independent conjugate processes is that an ergodic--like property of $\left(\xi_t\right)$ is inherited by $\left(X_\tau\right)$. This is the content of Theorem~\ref{thm:WLLN}.

\begin{theorem}\label{thm:WLLN}
Let $\left(\xi_t,\,X_\tau\right)$ be a {cyclic--independent} conjugate process. Assume $\left(\xi_t\right)$ is ergodic, in the sense that
\begin{equation*}
\lim_{n\rightarrow\infty}\frac{1}{n}\sum_{t=0}^{n-1}\xi_t = \E\xi_0\qquad\mbox{almost surely}.
\end{equation*}
Then, for any sequence $\left(\tau_i\right)$ with $\tau_i\in\left[i,\,i+1\right)$ and any continuous bounded function $f\colon\mathbb{R}\rightarrow\mathbb{R}$, it holds that
\begin{equation}\label{thm:LLN-basic}
\lim_{n\rightarrow\infty}\frac{1}{n}\sum_{i=0}^{n-1}f\circ X_{\tau_i} = \E f\circ X_0
\end{equation}
in probability.
\end{theorem}
With a little additional effort, Theorem~\ref{thm:WLLN} can be adapted to include other situations of interest -- for example, when $f$ is the identity function, when the intra--cycle sample sizes are larger than $1$, etc. Theorem~\ref{thm:WLLN} is important because it establishes a connection between the asymptotic behavior of time averages of the latent process and of the observable process. It is also important for in its proof a general method is presented for transferring ergodicity of $\left(\xi_t\right)$ to $\left(X_\tau\right)$. Variations of this idea are employed in many of the proofs that we give here. We now turn our attention to a characterization of the random measures $\xi_t$ by a spectral representation in a suitable Hilbert space.

\subsection{Hilbert space embedding and spectral representation}\label{sec:Hilbert-embedding}

A key feature of considering the random probability measures $\xi_t$ is that they can be embedded in a separable Hilbert space, in which they are characterized by a specific spectral representation. This spectral representation is described by a sequence of $\ell^2$ random elements and some functional parameters which are the eigenfunctions of a covariance operator.

Let $\mu$ be a finite Borel measure on $\mathbb{R}$. Let us say that $\mu$ is a \emph{diffraction} of the process $\left(\xi_t\right)$ if it is equivalent to Lebesgue measure on some interval containing the support of the measure $\E\xi_0$. The terminology is justified in Theorem~\ref{thm:spectral-representation-F_t}. Denote by $\langle\cdot,\cdot\rangle_\mu$ and $\norm{\cdot}_\mu$ the usual inner-product and norm in $L^2\left(\mu\right)$, respectively. Now define, for $x\in\mathbb{R}$,
\begin{equation}
F_t\left(x\right) := \xi_t\left(-\infty,x\right].
\end{equation}
Similarly, set
\begin{equation}
\E F_0\left(x\right) := \E\xi_0\left(-\infty,x\right].\label{eq:mean-function}
\end{equation}
By Lemma~\ref{thm:random-measure-basics} in the Appendix, $F_t\left(x\right)$ is a real random variable whose expectation equals $\E F_0\left(x\right)$. A bit more can be said.

\begin{lemma}\label{thm:hilbert-embedding-F_t}
If $\mu$ is a diffraction of $\left(\xi_t\right)$, then $\left(F_t:\,t=0,1,\dots\right)$ is a stationary sequence of random elements in $L^2\left(\mu\right)$. Moreover, the Bochner expectation of $F_t$ is equal to $\E F_0$, for all $t$.
\end{lemma}

As a random element in a separable Hilbert space, each $\xi_t$ most certainly admits a spectral representation as a series expansion with random scalar coefficients. We shall be interested in one such particular expansion. In that direction, for each $k=0,1,\dots$, let $C_k$ denote the $k$--th lag autocovariance function of $\left(F_t\right)$, that is, for $x,y\in\mathbb{R}$, put
\begin{equation}
C_k\left(x, y\right) := \Cov\left(F_0\left(x\right),\,F_k\left(y\right)\right), \label{eq:C_k}
\end{equation}
and introduce the operators $C_k^\mu$ acting on $L^2\left(\mu\right)$ defined by
\begin{equation}
C_k^\mu f\left(x\right) := \int C_k\left(x,y\right)f\left(y\right)\,\mu\left(dy\right).
\end{equation}

Recall that $C_0^\mu$ is positive and trace-class. The importance of the Hilbert space embedding, Lemma \ref{thm:hilbert-embedding-F_t}, lies in the spectral representation of $\left(F_t\right)$ as stated in Theorem~\ref{thm:spectral-representation-F_t}. This representation is related -- albeit in a slightly different guise -- to the well known Karhunen--Loève expansion (of each $F_t$ seen as a process $x\mapsto F_t\left(x\right)$). Before stating the theorem, let us establish some additional notation. In what follows, we shall write
\begin{equation*}\label{eq:DIMENSION}
	d := \rank{\left(C_0^\mu\right)} \leq\infty,
\end{equation*}
and we let $\left(\lambda_j^\mu:\,j\in \mathbb{N}\right)$ denote the non--increasing sequence of eigenvalues of $C_0^\mu$ (with repetitions, if any). To ensure that the sequence $\left(\lambda_j^\mu\right)$ is well defined we adopt the convention that it contains zeros if and only if $C_0^\mu$ is finite--rank. We define the eigenfunctions of $C_0^\mu$ via the equations
\begin{equation*}
	C_0^\mu \varphi_j^\mu = \lambda_j^\mu\varphi_j^\mu,\qquad j\in \mathbb{N},
\end{equation*}
and assume that the set $\left\{\varphi_j^\mu:\,j\in \mathbb{N}\right\}$ is orthonormal in $L^2\left(\mu\right)$.

\begin{theorem}\label{thm:spectral-representation-F_t}
Define the scalar random variables $Z_{tj}^\mu$ via
	\begin{equation}\label{eq:Z_tj-definition}
		Z_{tj}^\mu := \langle F_t - \E F_0, \varphi_j^\mu\rangle_\mu,\qquad j\in \mathbb{N},\quad t=0,1,\dots
	\end{equation}
	If $\mu$ is a diffraction of $\left(\xi_t\right)$, then
	\begin{enumerate}[label={(\textit{\roman*}})]
		\item $\E Z_{tj}^\mu = 0$ for all $t$ and all $j$;
		\item $\Var\left(Z_{tj}^{\mu}\right) = \lambda_j^\mu$ for all $t$ and all $j$;
		\item For all $t$, $\Cov\left(Z_{ti}^\mu, Z_{tj}^\mu\right) = 0$ whenever $i\neq j$;
		\item for each $j$, the function $\varphi_j^\mu$ is bounded and càdlàg.
	\end{enumerate}
Moreover, for each $t$ the expansion
	\begin{equation}\label{eq:F_t-spectral-representation-C_0}
		F_t = \E F_0 + \sum_{j = 1}^\infty Z_{tj}^\mu\varphi_j^\mu
	\end{equation}
holds in $L^2\left(\mu\right)$, almost surely.
\end{theorem}

\begin{remark}
The reader should contrast Theorem~\ref{thm:spectral-representation-F_t}, which states the almost sure convergence
$
	\lim_{N\rightarrow\infty}\int \big\{F_t\left(x\right) - \E F_0\left(x\right) -\sum_{j=1}^N Z_{tj}^\mu\varphi_j\left(x\right)\big\}^2\,\mu\left(dx\right) = 0,
$
with the Karhunen--Lo\`{e}ve Theorem, which states that
\begin{equation*}
	\lim_{N\rightarrow \infty}\sup_{x\in K}\E\big\{ \big[F_t\left(x\right) - \E F_0\left(x\right) -\sum_{j=1}^N Z_{tj}^\mu\varphi_j\left(x\right)\big]^2 \big\} = 0
\end{equation*}
for compact subsets $K\subset\mathbb{R}$. The Karhunen--Lo\`{e}ve Theorem also relies on some assumptions that may not hold in our context, namely \begin{enumerate*}[label={\textit{(\roman*)}}]\item that the map $\left(\omega,x\right)\mapsto F_t^\omega\left(x\right)$ is measurable, and \item that the covariance function $C_0$ is continuous.\end{enumerate*} Although the latter hypotheses can be relaxed, in the framework of random elements in Hilbert space Theorem~\ref{thm:spectral-representation-F_t} is more natural. It is also important to notice that the embedding of a random probability measure into the Hilbert space $L^2\left(\mu\right)$ induces random coefficients which are uniformly bounded and mutually dependent. Indeed, by the Pythagorean identity, we have that
	\begin{equation*}
	\sum_j \abs{Z_{tj}}^2 = \norm{F_t - \E F_0}^2\leq \mu\left(\mathbb{R}\right).
	\end{equation*}
Notice however that any dependence between the random variables, say, $Z_{t1}^\mu$ and $Z_{t2}^\mu$ is necessarily nonlinear since they are uncorrelated.
\end{remark}

\bigskip

Although the choice of the measure $\mu$ may seem quite arbitrary at first, the following proposition shows that, as long as one takes $\mu$ to be a diffraction of $\left(\xi_t\right)$, such choice does not matter. This implies, in particular, that the representation~\eqref{eq:F_t-spectral-representation-C_0} completely characterizes the process $\left(\xi_t\right)$.

\begin{proposition}\label{thm:mu-nu-equivalence}
	Let $\mu$ and $\nu$ be any two diffractions of $\left(\xi_t\right)$. Then $\rank\left(C_0^\mu\right) = \rank\left(C_0^\nu\right)$, and the $L^2\left(\mu\right)$--closed linear span of $\left\{\varphi_j^\nu:\,j \in \mathbb{N}\right\}$ and of $\left\{\varphi_j^\mu:\,j \in \mathbb{N}\right\}$ coincide.
\end{proposition}

In view of Theorem \ref{thm:spectral-representation-F_t} and Proposition \ref{thm:mu-nu-equivalence}, we shall henceforth suppress the measure $\mu$ from notation, writing for instance $\langle\cdot,\cdot\rangle$ in place of $\langle\cdot,\cdot\rangle_\mu$, etc. In particular, we write
\begin{equation}
F_t = \E F_0 + \sum_{j = 0}^\infty Z_{tj}\varphi_j
\end{equation}
instead of \eqref{eq:F_t-spectral-representation-C_0}. Theorem~\ref{thm:spectral-representation-F_t} and Proposition~\ref{thm:mu-nu-equivalence} say that the latent distribution process $\left(\xi_t\right)$ is characterized, via the $L^2\left(\mu\right)$--embedding, by the \emph{mean function} $\E F_0$, the \emph{functional parameters} $\left\{\varphi_j:\,j\in \mathbb{N}\right\}$ and by the $\ell^2$--valued stochastic process $\left(\boldsymbol{Z}_t:\,t=0,1,\dots\right)$, where $\boldsymbol{Z}_t := \left(Z_{tj}:\,j\in \mathbb{N}\right)$. It is important to notice that whenever $d<\infty$, the process $\left(\boldsymbol{Z}_t\right)$ can be seen as $\mathbb{R}^d$--valued. As we shall argue, in statistical applications one is mainly interested in this scenario, where the dynamic aspects of $\left(\xi_t\right)$ are driven by a finite dimensional vector process. A central inferential objective in this setting is to estimate the dimension $d$, the functional parameters $\E F_0, \varphi_1,\dots,\varphi_d$, and to recover the latent process $\left(\boldsymbol{Z}_t:\,t=0,1,\dots \right)$, based on sampling the process $\left(X_\tau\right)$ only.

\section{Inference}\label{sec:inference}
In this section we present an inference theory in the framework of conjugate processes. First we provide an alternate spectral characterization of the latent distribution process, after which we study the asymptotic properties of proposed estimators.

\subsection{Dynamic spectral representation of the latent process}\label{sec:dyn-repres}
From now on $p$ is some fixed integer with $1\leq p < n$ -- see the remark following Corollary~\ref{thm:finite-rank-equivalence2} for a discussion. Abusing a little on notation, for each $t$ let $X_{it}$, $i=1,\dots,q_t$ denote some observations of the process $\left(X_\tau\right)$ in cycle $t$. We define a \emph{sampling scheme} in terms of the collection $\left\{X_{it}\right\}_{it}$. We are being rather loose in the definition but the meaning should be evident; typically one has $X_{it} = X_{t+k/q_t}$ with $k=0,1,\dots,q_t\!-\!1$ but it is not necessarily so. Heuristically, whatever the generating process may be, we consider that one can only observe the data $\left(X_{it}:\,i=1,\dots,q_t,\,t=1,\dots,n\right)$. The underlying structure that characterizes the process $\left(\xi_t\right)$ -- namely, the expectation $\E F_0$, the eigenfunctions $\left\{\varphi_j:\,j\in \mathbb{N}\right\}$ with their associated eigenspace, and the process $\left(\boldsymbol{Z}_t\right)$ -- ought to be recovered from these data alone.

In this direction, let $\widehat{F}_t$ denote the empirical cumulative distribution function of the observations $X_{1t},X_{2t},\dots,X_{q_t,t}$, that is,
\begin{equation*}
\widehat{F}_t\left(x\right):=\frac{1}{q_t}\sum_{i=1}^{q_t}\mathbb{I}_{\left[X_{it}\leq x\right]}, \qquad x\in\mathbb{R}.
\end{equation*}
It is easily established that $\widehat{F}_t$ defines a random element in $L^2\left(\mu\right)$. Writing
\begin{equation}\label{eq:empirical-df-equals-F-plus-error}
	\widehat{F}_t = F_t + \varepsilon_t,
\end{equation}
where $\varepsilon_t = \widehat{F}_t - F_t$ \emph{by tautology}, we obtain the following properties.
\begin{lemma}\label{thm:epsilon-is-well-behaved}
Let $\left(\xi_t,X_\tau\right)$ be a cyclic--independent conjugate process. Then the following holds.
\begin{enumerate}[label={(\textit{\roman*}})]
	\item $\E\left[\varepsilon_t\left(x\right)\,\vert\,\Xi\right] = 0$ for all $t$ and all $x\in \mathbb{R}$;\label{thm:epsilon-is-well-behaved-item1}
	\item $\E\left[F_t\left(x\right)\varepsilon_{t+k}\left(y\right)\,\vert\,\Xi\right]=0$ for all $t$, all integers $k$ and all $x,y\in \mathbb{R}$;\label{thm:epsilon-is-well-behaved-item2}
	\item $\E\left[\varepsilon_t\left(x\right)\varepsilon_{t+k}\left(y\right)\,\vert\,\Xi\right]=0$ for all $t$ and all $x,y\in \mathbb{R}$ provided $k\neq 0$.\label{thm:epsilon-is-well-behaved-item3}
\end{enumerate}
\end{lemma}
In summary, Lemma~\ref{thm:epsilon-is-well-behaved} can be interpreted as saying that the intra--cycle empirical \textsc{cdf}s of conjugate processes (seen as random elements in $L^2\left(\mu\right)$) are decomposable as `underlying, true \textsc{cdf}' plus `noise'. Indeed, by iterated expectations, the equalities stated in Lemma~\ref{thm:epsilon-is-well-behaved} hold unconditionally, and thus $\left(\varepsilon_t\right)$ is a sequence of centered random elements in $L^2\left(\mu\right)$, strongly orthogonal to the process $\left(F_t\right)$, whose lagged covariance operators are all identically zero. Notice however that in general $\left(\varepsilon_t\right)$ is not \emph{white} noise since when $k=0$ in item~\ref{thm:epsilon-is-well-behaved-item3} the covariances may depend on $t$. The importance of Lemma~\ref{thm:epsilon-is-well-behaved} lies on the fact that it sheds some light on how to tackle a fundamental difficulty that rises when one is considering estimation of the eigenfunctions $\varphi_j$, and recovery of the process $\left(\boldsymbol{Z}_t\right)$. The issue is that, since the distribution functions $F_t$ are not observable, direct estimation of the covariance operator becomes spoiled. To see why this is so, for $0\leq k\leq p$ and $x,y\in\mathbb{R}$, define
\begin{equation}
\widehat{C}_{k}\left(x,y\right):=\frac{1}{n-p}\sum_{t=1}^{n-p}\big(\widehat{F}_{t}\left(x\right)-\widehat{\E}F_0\left(x\right)\big)\big(\widehat{F}_{t+k}\left(y\right)-\widehat{\E}F_0\left(y\right)\big),\label{eq:C_k_hat}
\end{equation}
where
\begin{equation}
\widehat{\E}F_0\left(x\right):=\frac{1}{n}\sum_{t=1}^{n}\widehat{F}_{t}\left(x\right).\label{eq:mu_hat}
\end{equation}
It is clear that $\widehat{C}_{0}$ is generally an inconsistent estimator of $C_{0}$. Indeed, $\widehat{C}_0$ estimates $\Cov\big(\widehat{F}_{t}\left(x\right),\widehat{F}_{t}\left(y\right)\big)$ which is equal to $C_{0}\left(x,y\right) + \Cov\left(\varepsilon_{t}\left(x\right),\varepsilon_{t}\left(y\right)\right)$, and as a rule the second term in the latter sum is not identically zero. For integers $k\neq 0$, however, Lemma~\ref{thm:epsilon-is-well-behaved} ensures that the equality $\Cov\big(\widehat{F}_{t}\left(x\right),\widehat{F}_{t+k}\left(y\right)\big)=C_{k}\left(x,y\right)$ holds, and hence under mild assumptions $\widehat{C}_k$ is a consistent estimator of $C_k$. Thus, while one is faced with an important drawback when considering the naïve approach of estimating the functions $\varphi_j$ and the process $\left(\boldsymbol{Z}_t\right)$ directly through estimation of $C_0^\mu$, the alternative of estimating the \emph{lagged} autocovariance operators $C_k^\mu$ seems promising.

Now let $R^{\mu}$ be the operator acting on $L^2\left(\mu\right)$ defined by
\begin{equation*}
R^{\mu}f\left(x\right):=\int R_\mu\left(x,y\right)f\left(y\right)\,\mu\left(dy\right),
\end{equation*}
where, for $x,y\in\mathbb{R}$, we let
\begin{equation}
R_\mu\left(x,y\right):=\sum_{k=1}^{p}\int C_{k}\left(x,z\right)C_{k}\left(y,z\right)\mu\left(dz\right).\label{eq:R}
\end{equation}
Clearly, $R^\mu = \sum_{k=1}^p C_k^\mu C_k^{\mu*}$, where the $*$ denotes adjoining, and thus $R^\mu$ is a positive operator. As was the case with the operator $C_0^\mu$, we shall establish some further notation before proceeding. Let
\begin{equation*}\label{eq:DIMENSION-Rmu}
	d' := \rank{\left(R^\mu\right)} \leq\infty,
\end{equation*}
and let $\left(\theta_j:\,j\in \mathbb{N}\right)$ denote the non--increasing sequence of eigenvalues of $R^\mu$ (with repetitions, if any). Again for well-definiteness we adopt the convention that the sequence $\left(\theta_j\right)$ contains zeros if and only if $d'<\infty$. We define the eigenfunctions of $R^\mu$ via the equations
\begin{equation*}
	R^\mu \psi_j = \theta_j\psi_j,\qquad j\in \mathbb{N},
\end{equation*}
and assume that the set $\left\{\psi_j:\,j\in \mathbb{N}\right\}$ is orthonormal in $L^2\left(\mu\right)$.

It is straightforward to show that $\Ran\left(R^\mu\right)\subset \overline{\Ran\left(C_0^\mu\right)}$. Let us say that any two operators $T$ and $S$ are \emph{range--equivalent} if $\overline{\Ran\left(T\right)} = \overline{\Ran\left(S\right)}$. Thus $R^\mu$ and $C_0^\mu$ are range--equivalent if and only if the inclusion $\overline{\Ran\left(R^\mu\right)}\supset {\Ran\left(C_0^\mu\right)}$ holds. If this is the case, we have the following.
\begin{theorem}\label{thm:spectral-representation-F_t-Rmu}
Assume that $R^\mu$ and $C_0^\mu$ are range--equivalent. Define the scalar random variables $W_{tj}$ via
	\begin{equation}\label{eq:W_tj-definition}
		W_{tj} := \langle F_t - \E F_0, \psi_j\rangle,\qquad j\in \mathbb{N},\quad t=0,1,\dots
	\end{equation}
	If $\mu$ is a diffraction of $\left(\xi_t\right)$, then
	\begin{enumerate}[label={(\textit{\roman*}})]
		\item $\E W_{tj} = 0$ for all $t$ and all $j$;
		\item $\Var\left(W_{tj}\right) = \sum_{i = 1}^\infty \lambda_i \langle\psi_j, \varphi_i\rangle^2$ for all $t$ and all $j$;
		\item for each $j$, the function $\psi_j$ is bounded.
	\end{enumerate}
Moreover, for each $t$ the expansion
	\begin{equation}\label{eq:F_t-spectral-representation}
		F_t = \E F_0 + \sum_{j = 0}^\infty W_{tj}\psi_j
	\end{equation}
holds in $L^2\left(\mu\right)$, almost surely.
\end{theorem}

Introducing the operator $R^\mu$ is justified in the same fashion as in Bathia et al.~\cite{bathia2010identifying}, and has an inferential motivation. In view of the above comments, the aim is to obtain representation~\eqref{eq:F_t-spectral-representation} as an alternative to \eqref{eq:F_t-spectral-representation-C_0}, the strategy thus becoming to estimate $R^\mu$, its associated eigenvalues and eigenfunctions, and to use the latter to recover the $\ell^2$--valued time series $\left(\boldsymbol{W}_t\right)$, where $\boldsymbol{W}_t:=\left(W_{tj}:\, j\in \mathbb{N}\right)$. Notice that, as was the case with $\left(\boldsymbol{Z}_t\right)$, whenever $C_0^\mu$ is finite--rank, the process $\left(\boldsymbol{W}_{t}\right)$ can be seen as being $\mathbb{R}^{d}$--valued, as in this setting the sum in $\eqref{eq:F_t-spectral-representation}$ has finitely many terms. For the above approach to make sense, however, one must assume that the $L^2\left(\mu\right)$--closures of $\Ran\left(C_k^\mu\right)$ and of $\Ran\left(C_0^\mu\right)$ coincide. When $C_0^\mu$ is finite--rank, we have the following criteria.

\begin{lemma}\label{thm:finite-rank-equivalence1}
Assume $d < \infty$. Then the operators $C_k^\mu$ and $C_0^\mu$ are range--equivalent if and only if the $d\times d$ matrix $\left(\E Z_{0i} Z_{kj}\right)_{ij}$ is of rank $d$.
\end{lemma}

\begin{corollary}\label{thm:finite-rank-equivalence2}
Assume $d < \infty$. If there is an integer $k\geq 1$ such that the matrix $\left(\E Z_{0i} Z_{kj}\right)_{ij}$ is of rank $d$, then, provided $p$ is large enough, the operators $R^\mu$ and $C_0^\mu$ are range--equivalent.
\end{corollary}

\begin{remark}
The condition that, for some $k$, the matrix $\left(\E Z_{0i} Z_{kj}\right)_{ij}$ is of rank $d$ is easier to appreciate in the case where $F_t - \E F_0$ lies in a one-dimensional subspace of $L^2\left(\mu\right)$, that is the case $d=1$. In this setting the matrix $\left(\E Z_{0i}Z_{kj}\right)_{ij}$ is indeed a scalar, and the condition that it is full-rank for some $k\geq 1$ means that the univariate time series $\left(Z_{t1}\right)$ is correlated at some lag $k$. Thus assuming that $\overline{\Ran\left(R^\mu\right)} = \overline{\Ran\left(C_0^\mu\right)}$ amounts to a requirement that the sequence $\left(F_t\right)$ displays `enough' dependence. Regarding the integer $p$, it is introduced in connection with the definition of the operator $R^\mu$: summing the lagged autocovariance operators up to $C_p^\mu$ is a way to ensure that $R^\mu$ captures the dependence structure of the process $\left(F_t\right)$, which in real data is likely to be concentrated on the first few lags and then decay. Said another way, in applications $p$ should be chosen by the statistician having in mind the fact that the precise values of $k$ for which the matrix $\left(\E Z_{0i} Z_{kj}\right)_{ij}$ is full rank are generally unknown. Using some of the lagged ${C}_k^\mu$ in the definition of ${R}^\mu$ is a parsimonious way to overcome this lack of knowledge.
\end{remark}

In view of the above discussion, in the remainder of the text we shall make the following assumption.
\begin{description}[leftmargin=!]
	\item[Assumption~R] \emph{The operators $R^\mu$ and $C_0^\mu$ are range--equivalent}.
\end{description}
Notice that Assumption~R \emph{rules out the possibility of $\left(\xi_t\right)$ being an independent sequence}. As previously argued, in statistical applications one is mostly interested in the scenario where $d < \infty$. This hypothesis relates to functional \textsc{pca} and identification of finite dimensionality in functional data. See Hall and Vial~\cite{hall2006assessing} and Bathia et al.~\cite{bathia2010identifying} for a discussion. In this setting, as mentioned, the dynamic aspects of $\left(\xi_t\right)$ are entirely determined by the $\mathbb{R}^d$--valued process $\left(\boldsymbol{Z}_t\right)$. But also, and this is the crucial point, Theorem~\ref{thm:spectral-representation-F_t-Rmu} ensures that these dynamic aspects are also determined by $\left(\boldsymbol{W}_t\right)$, as long as $\rank\left(R^\mu\right) = \rank\left(C_0^\mu\right)$. From either of these finite dimensional processes, the dynamics of the (in principle) infinite dimensional $\left(F_t\right)$ can be studied.

\subsection{Estimation and asymptotic properties}
Unfortunately, neither the process $\left(\boldsymbol{Z}_t\right)$ nor $\left(\boldsymbol{W}_t\right)$ is observable, not to mention the fact that the operator $R^\mu$ is unknown. Indeed in a first stage all one observes is a sample $\left(X_{1t},\dots,X_{q_t,t}\right)$ of the process $\left(X_\tau\right)$ in each cycle $t$, and the associated empirical distribution functions $\widehat{F}_t$. However, under suitable conditions one can hope to recover the underlying structure which characterizes the latent process $\left(\xi_t\right)$. In this direction, define
\begin{equation}
\widehat{R}_\mu\left(x,y\right)=\sum_{k=1}^{p}\int\widehat{C}_{k}\left(x,z\right)\widehat{C}_{k}\left(y,z\right)\mu\left(dz\right),\label{eq:M_hat}
\end{equation}
and let $\widehat{R}^\mu$ be the integral operator with kernel $\widehat{R}_\mu$ and non--increasing sequence of eigenvalues $\big(\widehat{\theta}_j:\,j\in\mathbb{N}\big)$, with repetitions. It is straightforward to show that $\widehat{R}^\mu$ is a finite--rank operator, of rank, say, $d_n$, with $d_n \leq n-p$, and thus we have $\widehat{\theta}_1\geq\dots\geq\widehat{\theta}_{d_n}>0=\widehat{\theta}_{d_n+1} = \widehat{\theta}_{d_n+2} = \cdots$. Define the eigenfunctions of $\widehat{R}^\mu$ via the equations
\begin{equation*}
	\widehat{R}^\mu \widehat{\psi}_j = \widehat{\theta}_j\widehat{\psi}_j,\qquad j\in \mathbb{N},
\end{equation*}
and assume that the set $\big\{\widehat{\psi}_j:\,j\in \mathbb{N}\big\}$ is orthonormal in $L^2\left(\mu\right)$. In the Appendix~\ref{sec:estimation} below we provide a straightforward estimation procedure for computing these quantities, which relies on simple matrix analysis. Now put, for $j\in\mathbb{N}$ and $1\leq t\leq n$,
\begin{equation}\label{eq:W-hat-definition}
	\widehat{W}_{tj} := \big\langle \widehat{F}_t - \widehat{\E}F_0, \widehat{\psi}_j \big\rangle.
\end{equation}
In what follows we study the asymptotic properties of the proposed estimators.

\bigskip

Let us first consider estimation of ${\E}F_0$. An important property of conjugate processes is that a \textsc{lln} for $\big(\widehat{F}_t\big)$ holds under weak assumptions on the sampling scheme and on the latent distribution process $\left(\xi_t\right)$.

\begin{theorem}\label{thm:LLN-for-Fhat}
Let $\left(\xi_t,X_\tau\right)$ be a cyclic--independent conjugate process, and let $\mu$ be a diffraction of $\left(\xi_t\right)$. If the sequence $\left(\xi_t\right)$ is $\Prob$-ergodic, in the sense that $\lim_{n\rightarrow\infty}n^{-1}\sum_{t=1}^n\xi_t = \E\xi_0$ in probability, then
\begin{enumerate}[label={(\textit{\roman*}})]
	\item $\Vert n^{-1}\sum_{t=1}^n F_t - \E F_0\Vert = o_\Prob\left(1\right)$;\label{thm:LLN-for-Fhat-item-i}
	\item $\Vert \widehat{\E}F_0 - \E F_0\Vert = o_\Prob\left(1\right)$.\label{thm:LLN-for-Fhat-item-ii}
\end{enumerate}
If moreover $\norm{n^{-1}\sum_{t=1}^n F_t - \E F_0} = O_\Prob\left(n^{-1/2}\right)$, then it holds that $\big\Vert\widehat{\E}F_0 - \E F_0\big\Vert = O_\Prob\left(n^{-1/2}\right)$.
\end{theorem}

\bigskip

Regarding estimation of the eigenfunctions $\psi_j$, further assumptions on the process $\left(\xi_t, X_\tau\right)$ may be needed. We establish consistency of $\widehat{\psi}_j$ in two different flavors: first, without a rate and assuming little else than convergence $n^{-1}\sum_{t=1}^n F_t\left(x\right) F_{t+k}\left(y\right) \rightarrow \E F_0\left(x\right)\E F_k\left(y\right)$. Then, by imposing a $\psi$--mixing condition\footnote{See Bradley~\cite{bradley2005basic} for a discussion and definitions.} on $\big(\widehat{F}_t\big)$, we are able to derive consistency with a rate. In any case, this is straightly related to convergence $\widehat{R}^\mu\rightarrow R^\mu$ in the Hilbert--Schmidt\footnote{Recall that an operator acting on $L^2\left(\mu\right)$ is said to be a \emph{Hilbert--Schmidt operator} if and only if it is an integral operator whose kernel lies in $L^2\left(\mu\otimes\mu\right)$. The \emph{Hilbert--Schmidt norm} $\norm{\cdot}_{HS}$ of such an operator is the $L^2\left(\mu\otimes\mu\right)$ norm of its kernel.} norm $\norm{\cdot}_{HS}$. The additional assumptions in Theorems~\ref{thm:LLN-for-psihat-no-rate} and \ref{thm:LLN-for-psihat} below can, in a sense, be understood as asking that the process $\left(\xi_t\right)$ does not display `too much' dependence, in contrast with the requirement that $R^\mu$ and $C_0^\mu$ are range--equivalent, which imposes that the latent distribution process displays `enough' dependence.

\begin{theorem}\label{thm:LLN-for-psihat-no-rate}
	In addition to the conditions of Theorem~\ref{thm:LLN-for-Fhat}, assume that, for $x,y\in\mathbb{R}$ and $1\leq k\leq p$, one has
	\begin{equation}\label{eq:auto-covariance-consistency}
		\frac{1}{n}\sum_{t=1}^n F_t\left(x\right) F_{t+k}\left(y\right) = \E F_0\left(x\right) F_k\left(y\right) + o_\Prob\left(1\right),
	\end{equation}
	in $L^2\left(\mu\otimes\mu\right)$. Then,
	\begin{enumerate}[label={(\textit{\roman*}})]
		\item $\Vert \widehat{R}^\mu - R^\mu\Vert_{HS} = o_\Prob\left(1\right)$;
		\item $\sup_{j\in \mathbb{N}}\vert \widehat{\theta}_j - \theta_j\vert = o_\Prob\left(1\right)$.
	\end{enumerate}
	If moreover the nonzero eigenvalues of $R^\mu$ are all distinct, then
	\begin{enumerate}[label={(\textit{\roman*}}), start = 3]
	\item $\Vert \widehat{\psi}_j - \psi_j\Vert = o_\Prob\left(1\right)$, for each $j$ such that $\theta_j > 0$.
	\end{enumerate}
\end{theorem}
The condition in equation~\eqref{eq:auto-covariance-consistency} is imposed \emph{ad hoc}. For insights on Hilbertian processes $\left(F_t\right)$ for which it might hold, see Bosq~\cite{bosq2002estimation}. The following result is essentially a restatement of Theorem~1 in Bathia et al.~\cite{bathia2010identifying}, which gives sufficient conditions for $\sqrt{n}$-consistency of $\widehat{\psi}_j$.

\begin{theorem}\label{thm:LLN-for-psihat}
Let $\left(\xi_t,X_\tau\right)$ be a cyclic--independent conjugate process, and let $\mu$ be a diffraction of $\left(\xi_t\right)$. Assume that $\big(\widehat{F}_t\big)$ is a $\psi$--mixing sequence, with the mixing coefficients $\Psi\left(k\right)$ satisfying $\sum_{k=1}^\infty k\,\Psi^{1/2}\left(k\right) < \infty$. Then it holds that
\begin{enumerate}[label={(\textit{\roman*}})]
	\item $\Vert \widehat{R}^\mu - R^\mu\Vert_{HS} = O_\Prob\left(n^{-1/2}\right)$;
	\item $\sup_{j\in \mathbb{N}}\vert \widehat{\theta}_j - \theta_j\vert = O_\Prob\left(n^{-1/2}\right)$.
\end{enumerate}
If moreover the nonzero eigenvalues of $R^\mu$ are all distinct, then
\begin{enumerate}[label={(\textit{\roman*}}), start = 3]
	\item $\Vert\widehat{\psi}_j - \psi_j\Vert = O_\Prob\big(n^{-1/2}\big)$, for each $j$ such that $\theta_j > 0$.
\end{enumerate}
\end{theorem}

The assumptions in Theorem~\ref{thm:LLN-for-psihat} correspond to assumptions C1 and C3 from Theorem~1 in Bathia et al.~\cite{bathia2010identifying}. The requirement that $\theta_1 > \theta_2 > \cdots$ is a simplification and can be relaxed -- see Mas and Menneteau~\cite{mas2003perturbation} for a discussion. {Regarding the $\psi$--mixing assumption, notice that since it imposes restrictions on the process $\big(\widehat{F}_t\big)$ it will likely involve properties of both $F_t$ and $X_\tau$ jointly. It can be shown that, if the latent sequence $\left(\xi_t\right)$ satisfies the $\psi$--mixing assumption of the theorem, and if a condition slightly stronger than cyclic independence is imposed, then the sequence of empirical \textsc{cdf}s $\big(\widehat{F}_t\big)$ will inherit the $\psi$--mixing property. Introducing the definitions required for adequately tackling this topic would be too technical and beyond the scope of the present paper, and therefore we withdraw from the discussion.

\bigskip

\paragraph{Estimating the underlying $\ell^2$ dynamics} The following result provides a description of how $\widehat{W}_{tj}$ fluctuates around $W_{tj}$, and in particular it shows that there is a bound on how far these quantities can be one from another.
\begin{proposition}\label{thm:LLN_etahat}
Let $\left(\xi_t,X_\tau\right)$ be a cyclic--independent conjugate process. Then, for all $j\in\mathbb{N}$ and all $1\leq t \leq n$, it holds that
\begin{equation}\label{eq:eta-hat-asymptotics}
\widehat{W}_{tj} = W_{tj} + \langle \varepsilon_t,\psi_j\rangle + \widehat{\rho}_{j}
\end{equation}
almost surely, where $\widehat{\rho}_j$ is a random variable satisfying $\abs{\widehat{\rho}_j} \leq \Vert \widehat{\E}F_0 - \E F_0 \Vert + 2 \vert\mu\vert^{1/2} \Vert \widehat{\psi}_j - \psi_j \Vert$. Moreover, the innovations $\big(\langle\varepsilon_t, \psi_j\rangle:\,j\in\mathbb{N},\,t=0,1,\dots\big)$ are such that
\begin{enumerate}[label={(\textit{\roman*}})]
	\item $\E\left[\langle\varepsilon_t,\psi_j\rangle\,\vert\,\Xi\right] = 0$ for all $t$ and $j$.
	\item $\E\left[\langle\varepsilon_t,\psi_j\rangle\cdot\langle\varepsilon_s,\psi_i\rangle\,\vert\,\Xi\right] = 0$ for all $j,i$ and all $t\neq s$.
	\item $\E\left[W_{tj}\cdot\langle\varepsilon_s,\psi_i\rangle\,\vert\,\Xi\right] = 0$ for all $t,s$ and all $j,i$.
	\item $\vert\langle\varepsilon_t,\psi_j\rangle\vert \leq \Vert\widehat{F}_t - F_t\Vert$ for all $t$ and $j$.
\end{enumerate}
\end{proposition}

\begin{corollary}\label{thm:LLN2_etahat}
	In the conditions of Theorems~\ref{thm:LLN-for-Fhat} and \ref{thm:LLN-for-psihat}, it holds that
	\begin{enumerate}[label={(\textit{\roman*}})]
	\item $\E\big[\widehat{W}_{tj}\,\vert\, \Xi\big] = W_{tj} + O_\Prob\left(n^{-1/2}\right)$.
	\item $\vert\widehat{W}_{tj} - W_{tj}\vert \leq \Vert \widehat{F}_t - F_t \Vert + O_\Prob\left(n^{-1/2}\right)$.\label{thm:LLN2_etahat-item-ii}
\end{enumerate}
\end{corollary}

Proposition~\ref{thm:LLN_etahat} and Corollary~\ref{thm:LLN2_etahat} together say that, for each $j$, the sample paths of $\big(\widehat{W}_{tj}\big)$ and $\big(W_{tj}\big)$ are essentially tied together. Notice, however, that asymptotics on $n$ alone is not sufficient to provide a small bound on $\vert\widehat{W}_{tj} - W_{tj}\vert$: one should also control the size of the innovation sequence $\left(\langle \varepsilon_t,\psi_j\rangle\right)$, via the term $\Vert \widehat{F}_t - F_t\Vert$ which depends essentially on intra--cycle aspects of $\left(X_\tau:\,\tau\in \left[t,t+1\right)\right)$, particularly on the sample size, $q_t$.

In fact, ideally one would be interested in bounding the quantity $\max_{1\leq t\leq n}\vert \widehat{W}_{tj} - W_{tj}\vert$ by a small constant, that is, to make the sample paths of $\big(\widehat{W}_{tj}\big)$ and $\big(W_{tj}\big)$, $1\leq t\leq n$, uniformly close one to another. We now show that this is not always attainable, although it may be possible to achieve such an approximation `for a certain amount of time'. For convenience let $a_r^s\left(\delta\right) := \Prob\big(\max_{r\leq t\leq s}\Vert \widehat{F}_t - F_t\Vert > \delta\big)$, where $1\leq r \leq s\leq n$ and $\delta >0$; it is straightforward to show that, for each $j\in\mathbb{N}$ and arbitrary $\epsilon>0$, we have $\Prob\big(\max_{r\leq t\leq s}\vert \widehat{W}_{tj} - W_{tj}\vert > \delta + \epsilon\big) \leq a_r^s\left(\delta\right) + \epsilon$, provided the sample size is large enough. The issue thus becomes to make $a_r^s\left(\delta\right)$ as small as possible, with $\delta$ as small as possible. We shall discuss the size of $a_r^s\left(\delta\right)$ in an ideal setting, for the sake of exposition. For simplicity, let $\abs{\mu} = 1$ so that $\Vert\widehat{F}_t - F_t\Vert \leq \sup_x\vert \widehat{F}_t\left(x\right) - F_t\left(x\right)\vert$. Assume that, within each cycle $t$, one can devise a sampling scheme such that, conditional on $\Xi$, an iid sample $\left(X_{1,t}, X_{2,t},\dots, X_{q_t, t}\right)$ of size $q_t$ from the distribution $\xi_t$ is attainable. Then, by the Dvoretzky–Kiefer-Wolfowitz (DKW) inequality we have, for each $t$ and each $\delta > 0$, that $\Prob\left(\big\Vert \widehat{F}_t - F_t\big\Vert > \delta\,\big\vert\,\Xi\right) \leq 2\exp\left(-2q_t\delta^2\right)$ almost surely. By taking expectations we see that $a_t^t\left(\delta\right)$ can be made small if $q_t$ can be taken large. This settles the question of approximating $W_{tj}$ by $\widehat{W}_{tj}$ individually. Assuming cyclic--independence in turn yields, by a straightforward calculation and the Bernoulli inequality, the upper bound $a_r^s\left(\delta\right) \leq 2\left(s-r+1\right)\exp\left(-2q_*\delta^2\right)$, where $q_* := \min_{1\leq t\leq n} q_t$. This shows that there is an interplay between the precision, $\delta + \epsilon$, the intra--cycle sample sizes (as measured by $q_*$), and the possibility of approximating the sample path of $\left(W_{tj}\right)$ uniformly for $t=r,\dots,s$ with such precision. On the one hand, a large sample size $n$ is desirable as a means to bound (in probability) the terms $\Vert \widehat{\E}F_0 - \E F_0\Vert$ and $\Vert\widehat{\psi}_j - \psi_j\Vert$ by $\epsilon$. Indeed, the constant $\epsilon$ can be taken arbitrarily small as long as the sample size is sufficiently large. On the other hand, since $q_*$ is bounded as $n$ increases, the possibility of approximating the latent $\left(W_{tj}\right)$ uniformly for $t=1,\dots,n$ becomes spoiled as $n$ increases: typically $a_1^n\left(\delta\right)$ cannot be made arbitrarily small for small values of $\delta$. In applied work, if the interest lies in studying properties of the latter time series, then a careful choice of $\delta$, and of the indices $r$ and $s$, may ensure that the sample paths of $\left(W_{tj}:\,t=r,\dots,s\right)$ can be approximated by $\big(\widehat{W}_{tj}:\,t=r,\dots,s\big)$ with precision $\delta + \epsilon$ and with a large probability. The above exposition can be generalized to a setting where the processes $\left(X_\tau\vert\Xi:\,\tau\in\left[t,t+1\right)\right)$ display a stronger dependence structure, by appealing to generalizations of the DKW inequality as discussed for instance in Dedecker and Merlev\`{e}de~\cite{dedecker2007empirical} and Kontorovich and Weiss~\cite{kontorovich2014uniform}.

\bigskip

\paragraph{Estimation of $F_t$} We now discuss the problem of proposing an estimator of $F_t$ (other than $\widehat{F}_t$) which takes into account the presented methodology. First, there is the issue of estimating the dimension $d$, which is unknown to the statistician. In that regard, we have the following.
\begin{proposition}\label{thm:d-hat-consistency}
	Assume $d < \infty$, and let $\widehat{d}:=\#\big\{j:\,\widehat{\theta}_j\geq a_n\big\}$. If $a_n\rightarrow 0$ and $na_n^2\rightarrow\infty$, then
	\begin{equation*}
		\Prob\big(\widehat{d}\neq d\big)\rightarrow 0
	\end{equation*}
	as $n\rightarrow\infty$.
\end{proposition}
Now let $\widehat{d}$ be a consistent estimator of $d$, and define the estimator
\begin{equation}\label{eq:F-tilde-definition}
\widetilde{F}_t:=\widehat{\E}F_0 + \sum_{j=1}^{\widehat{d}} \widehat{W}_{tj}\widehat{\psi}_j.
\end{equation}
Notice that it may happen that $\widetilde{F}_t$ is not a \textsc{cdf} (it will be if we set $\widehat{d}=d_n$). That is, even though $\widetilde{F}_t$ may be close to $F_t$ in the $L^2\left(\mu\right)$ norm, nothing grants that it will be nondecreasing or have its values strictly between 0 and 1. This is not a major issue, however. For example, in the context of nonparametric density estimation it is common to permit estimators which  in finite samples are not densities but which have good asymptotic properties. For ease of exposition, we shall consider the case where it is known that $d = 1$. This can be easily generalized to any $d < \infty$.
\begin{corollary}\label{thm:F-tilde-asymptotics}
	Let $d = \widehat{d} = 1$ and let $\widetilde{F}_t$ be defined as in \eqref{eq:F-tilde-definition}. In the conditions of Theorems~\ref{thm:LLN-for-Fhat} and \ref{thm:LLN-for-psihat}, it holds that
	\begin{equation}\label{eq:F-tilde-asymptotics}
		\widetilde{F}_t = F_t + \langle\varepsilon_t, \psi_1\rangle\psi_1 + O_\Prob\big(n^{-1/2}\big)
	\end{equation}
	in $L^2\left(\mu\right)$.
\end{corollary}
Expression \eqref{eq:F-tilde-asymptotics} should be contrasted with \eqref{eq:empirical-df-equals-F-plus-error}: whereas $\Vert\widehat{F}_t - F_t\Vert = \Vert\varepsilon_t\Vert$, we have that $\Vert\widetilde{F}_t - F_t\Vert \leq \vert\langle\varepsilon_t, \psi_1\rangle\vert + O_\Prob\big(n^{-1/2}\big)$, and since $\Vert\varepsilon_t\Vert^2 = \sum_{j=1}^\infty \vert\langle \varepsilon_t,\psi_j\rangle\vert^2$, one will typically have $\vert\langle\varepsilon_t, \psi_1\rangle\vert < \Vert\varepsilon_t\Vert$ unless we find ourselves in the very unlikely situation where $\varepsilon_t$ is orthogonal to all the $\psi_j$ with $j\geq2$. This means that, as long as the sample size $n$ is sufficiently large, the estimator $\widetilde{F}_t$ represents an improvement over $\widehat{F}_t$ (in the $L^2\left(\mu\right)$ sense). As seen in the simulation study below, such improvements may be substantial in certain scenarios.

\section{An example}\label{sec:example}

The aim of this section is to construct an illustrative example rather than providing a thorough simulation study. Here $\mu$ is Lebesgue measure restricted to the interval $I=\left[-1,1\right]$. Write $\psi_1\equiv \psi$, and likewise $W_{t1}\equiv W_t$. Assume $W_0, W_1,\dots $ is a stationary \textsc{ar}(1) process, $W_t = \alpha W_{t-1} + u_t$, where $\left(u_t\right)$ is some centered iid real sequence and $\abs{\alpha}<1$. Let $H$ be a fixed \textsc{cdf} concentrated on $I$ with $\int x\, dH\left(x\right)=0$, and let $\psi$ be some bounded function on $\left[-1,1\right]$ with $\psi\left(-1\right)=\psi\left(1\right)=0$. Now write $F_t\left(x\right) = H\left(x\right) + W_t\psi\left(x\right)$. A straightforward calculation yields $F_t\left(x\right)=\left(1-\alpha\right)H\left(x\right) + \alpha F_{t-1}\left(x\right) + u_t \psi\left(x\right)$, that is, $\left(F_t\right)$ is a linear process as well! Clearly, some restrictions on $\psi$ and on the process $\left(W_t\right)$ must be imposed to ensure that the $F_t$ are indeed \textsc{cdf}s, but we relegate the details on how to achieve this to the simulation exercise below. Assuming further that $\int x\,\psi\left(dx\right) = 0$ we obtain $\int x \,F_t\left(dx\right)=0$, and letting $\xi_t$ denote the measure corresponding to $F_t$, one sees that any process $\left(X_\tau:\,\tau\geq 0 \right)$ satisfying \eqref{eq:weakly-conjugate-model}, will be such that $\E\left[X_\tau \vert \xi_0, \xi_1,\dots\right]=0$. Also it is clear that $\E\xi_0$ is the measure corresponding to $H$. In the notation of the previous sections, the operator $C_0^\mu$ is seen to be of rank $d=1$, with $\psi$ being an eigenfunction associated to the eigenvalue $\lambda = \norm{\psi}^2\E W_0^2$. Also, since $\E W_0 W_1 \neq 0$, the hypothesis in Theorem~\ref{thm:spectral-representation-F_t-Rmu} is satisfied and thus $\psi$ is an eigenfunction of $R^\mu$ as well (with $p=1$), associated to the eigenvalue $\alpha^2\lambda^2$.

%In this context an important question that rises is whether the process just defined will induce a sequence $(\widehat{F}_t)$ satisfying the $\psi$--mixing condition of Theorem~\ref{thm:LLN-for-psihat}. By Proposition~\ref{thm:mixing-inheritance}, if \eqref{eq:strong-cyclic-independence} holds then it is sufficient that the \textsc{ar}(1) process $\left(W_t\right)$ satisfies the corresponding condition. However, recall that not every \textsc{ar} process is mixing: as shown for instance in Andrews~\cite{andrews1984nonstrong}, setting the regression coefficient $\alpha = 1/2$ and letting the $u_t$'s be Bernoulli random variables induces a process $\left(W_t\right)$ which fails to be strong mixing, which is a weaker condition than $\psi$--mixing. However, if the innovations are absolutely continuous r.v's, then the strong mixing property does hold -- see Bosq~\cite[p. 17]{bosq1998nonparametric} for example. An exhaustive discussion of the $\psi$--mixing property for autoregressive processes seems to be scarce in the literature; the interested reader may find a very technical account in Doukhan~\cite{doukhan1994mixing}.

\begin{figure}[h]
	\centering
	\caption{Residuals $\widehat{W}_t - W_t$ with (a) $n=100$, $q=100$; (b) $n=100$, $q=200$; (c) $n=200$, $q=100$; (d) $n=200$, $q=200$. Grey: sample path of $\widehat{W}_t$.}\label{fig:simu-etahat-minus-eta-ts}
	\includegraphics[scale = 1]{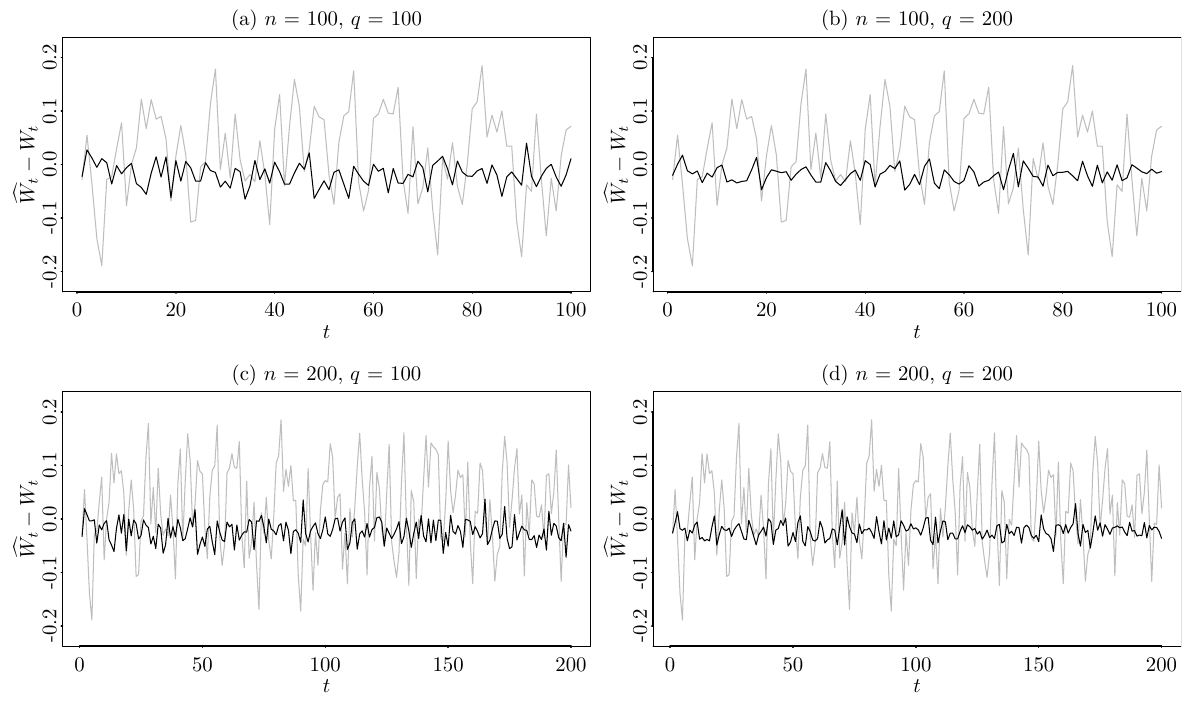}
\end{figure}

Narrowing a little further, let us consider the following special case of the above example. Let $H$ be the \textsc{cdf} corresponding to the uniform distribution over $\left[-1,1\right]$, and let $\psi\left(x\right):=\int_{-1}^x \left( 1/2 - \abs{v}\right)\,dv$. Let $W_t$ be a stationary \textsc{ar}(1) process as above, with the innovations $u_t$ being iid uniformly distributed over $\left[-1+\abs{\alpha}, 1-\abs{\alpha}\right]$. We may assume that the process $\left(u_t\right)$ is indexed for $t\in\mathbb{Z}$ and set $W_t = \sum_{k=0}^\infty \alpha^k u_{t-k}$. This model specification is easier to appreciate if we consider the derivatives of $H$ and $\psi$ over $\mbox{interior}\left(I\right)$, that is, we gain better insight if we differentiate $F_t$ and study the resulting equation, $f_t = h + W_t\psi'$, with $h\left(x\right)=\left(1/2\right)\mathbb{I}_{\left[-1,1\right]}\left(x\right)$ and $\psi'\left(x\right)=\left\{\left(1/2\right) - \abs{x}\right\}\mathbb{I}_{\left[-1,1\right]}\left(x\right)$. First notice that $\abs{W_t}\leq 1$ by construction. Now $f_t$ is a probability density function obtained by adding to the Uniform$\left[-1,1\right]$ density a random deformation where the deforming `parameter' is the function $\psi'$ and the random weights are given by the $W_t$ which lie in $\left[-1,1\right]$. The extreme cases correspond to $W_t=1$, in which case $f_t$ is the triangular distribution over $\left[-1,1\right]$, and to $W_t=-1$, in which case $f_t$ is a V-shaped distribution, $f_t\left(x\right)=\abs{x}\mathbb{I}_{\left[-1,1\right]}\left(x\right)$. Any possible realization of $f_t$ is thus a convex combination of the latter two densities. The interpretation is that $\psi$ adds mass to the center of the uniform distribution when $W_t>0$ and adds mass to the `tail' of that distribution when $W_t<0$. Observe that the proposed $\psi$ is not normalized, but this does not matter since the rescaling would be passed to the $W_t$'s.

To illustrate, we set $\alpha=0.5$ and generated a sample $F_1,\dots,F_{200}$ from the above model and then, for each $t$, we sampled the $X_{it}$, $i=1,\dots,200$, as independent draws from $F_t$. Sampling independently is a simplification but not inconsistent with the present framework, as it may be the case that the process $\left(X_\tau\right)$ admits an independent sampling scheme at each cycle. Next, we estimate $W_t$ and $\psi$ restricting the data set to $n\leq 200$ cycles and $q\leq 200$ intra--cycle observations. We consider the following configurations: \begin{enumerate*}[label={\textit{(\roman*)}}] \item $n=100$, $q=100$; \item $n=100$, $q=200$; \item $n=200$, $q=100$ and; \item full sample $n=200$, $q=200$. \end{enumerate*} Figure~\ref{fig:simu-etahat-minus-eta-ts} shows the residuals $\widehat{W}_t - W_t$ in each of these configurations. It is apparent that increasing the intra-cycle sample sizes will result in more accurate estimates for the $W_t$, as one would expect from Proposition~\ref{thm:LLN_etahat} and Corollary~\ref{thm:LLN2_etahat}. Figure~\ref{fig:simu-psihat-minus-psi} displays the true eigenfunction $\psi$ and the estimates $\widehat{\psi}$, together with the deviations $\widehat{\psi} - \psi$ for each one of the specifications (i)--(iv). These figures point to the fact that, although Theorems~\ref{thm:LLN-for-psihat-no-rate} and \ref{thm:LLN-for-psihat} ensure that asymptotics on $n$ will suffice for consistency of $\widehat{\psi}$, increasing the intra-cycle sample sizes may have a positive impact on estimation as well.

We also study the distributional aspects of the quantities $\Vert \widetilde{F}_t - F_t\Vert$ and $\Vert \widehat{F}_t - F_t\Vert$ in the above setting, with $q = 50$ and with $n \in\{100, 200\}$. Now, since \begin{enumerate*}[label={\textit{(\roman*)}}]\item the data generating process under consideration is such that each sample $X_{1,t},\dots,X_{q_t,t}$ is iid $\sim\xi_t$; \item $q_t = q$ for all $t$; and \item $\left(\xi_t\right)$ is strongly stationary, \end{enumerate*} it is easily seen that the distribution of the random variable $\Vert \widetilde{F}_t - F_t\Vert$ does not depend on $t$ (the same is true of $\Vert \widehat{F}_t - F_t\Vert$, which also does not depend on $n$), and thus it is sufficient to consider the Monte Carlo distribution of, say, $\Vert \widetilde{F}_1 - F_1\Vert$ and $\Vert \widehat{F}_1 - F_1\Vert$. Figure~\ref{fig:simu-boxplot} displays the boxplots corresponding to the Monte Carlo distribution of the latter quantities, obtained across $2000$ replications. These results further illustrate the improvements implied by Corollary~\ref{thm:F-tilde-asymptotics}.

\begin{figure}%[H]
	\centering
	\caption{True eigenfunction $\psi$ (dotted), estimated eigenfunction $\widehat{\psi}$ (solid) and deviation $\widehat{\psi} - \psi$ (dashed). (a) $n=100$, $q=100$; (b) $n=100$, $q=200$; (c) $n=200$, $q=100$; (d) $n=200$, $q=200$.}\label{fig:simu-psihat-minus-psi}
	\includegraphics{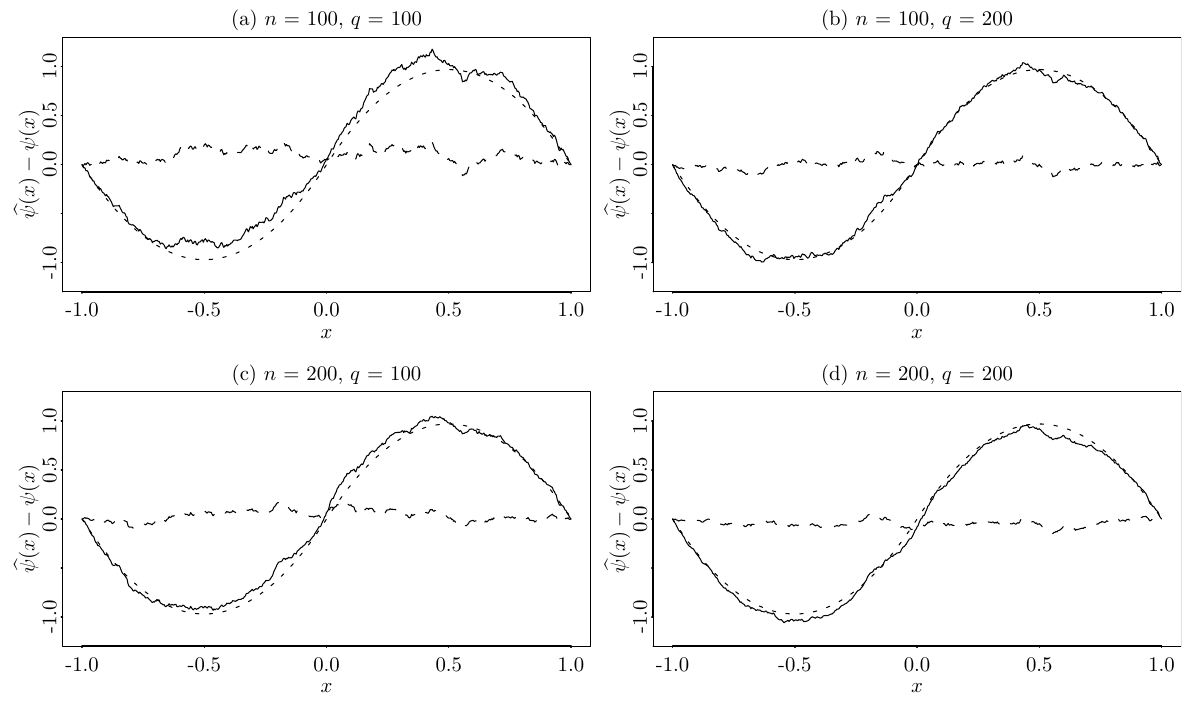}
\end{figure}

In this simulation study and in the empirical application below, all computational work was carried out through the software packages \texttt{R} and \texttt{Julia}.

\begin{figure}%[H]
	\centering
	\caption{Monte Carlo boxplot corresponding to: (left) $\Vert \widetilde{F}_1 - F_1\Vert$ with $q = 50$ and $n = 200$; (center) $\Vert \widetilde{F}_1 - F_1\Vert$ with $q = 50$ and $n = 400$; (right) $\Vert \widehat{F}_1 - F_1\Vert$ with $q = 50$.}\label{fig:simu-boxplot}
	\includegraphics{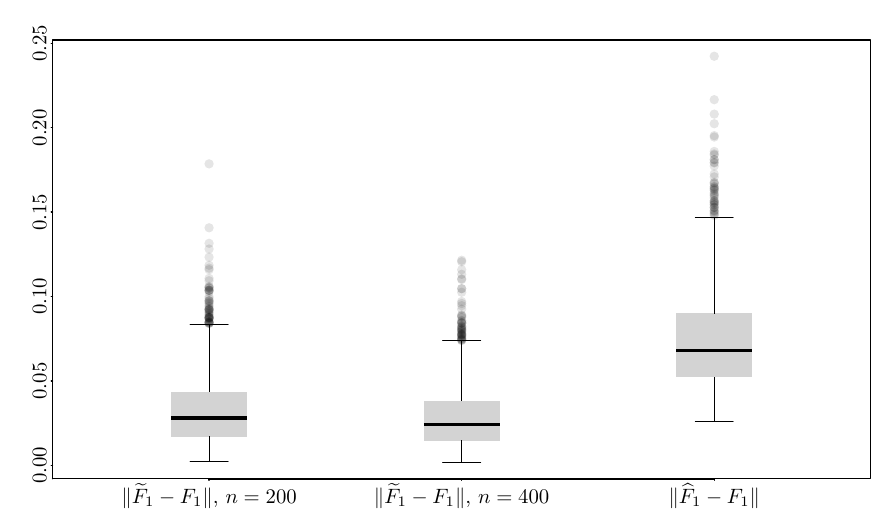}
\end{figure}

\section{Application to financial data}\label{sec:application}
We apply our methodology to forecast distributional aspects and risk in high frequency stock market trading. Our sample consists of 5--minute returns for the \textsc{itub4} asset prices; the raw data is available at the Bovespa ftp site (ftp://ftp.bmf.com.br/marketdata). \textsc{itub4} is the main asset in the composition of the Bovespa index. Our sample ranges from July 1st 2012 to April 30 2015, encompassing 719 business days. At each day $t$ the sample $X_{1t},\cdots,X_{q_t,t}$ consists of $q_t=79$ observations of the 5--minute return process, defined as the difference of logarithm prices over 5 minutes (with daily averages subtracted), ranging from 10:30 \textsc{am} to 5:00 \textsc{pm}. There are 3 carnival days during the sampling period, at which the intra--day sample sizes are $q_{170}=47$, $q_{433}=46$ and $q_{670}=47$ respectively. Our working assumption is that the $X_{it}$ are sampled from a conjugate process $\left(\xi_t,\,X_\tau\right)$, and thus we are assuming that, on day $t$ and conditional on $\left(\xi_t\right)$, the 5--minute returns share the same marginal distribution $\xi_t$. The empirical distribution functions of 5--minute returns for the first two days in our sample, $\widehat{F}_1$ and $\widehat{F}_2$, are plotted in Figure~\ref{fig:sample}. In what follows $\mu$ is the Laplace(0,1) distribution on the real line. In preliminary analyses (not reported here) we tested values of $p$ ranging from $p = 1$ to $p = 10$, without noticeable alterations in the obtained results. We thus set $p=5$ following Bathia et al.~\cite{bathia2010identifying}.
\begin{figure}%[H]
	\centering
	\caption{Empirical \textsc{cdf}s of 5--minute returns: days 1 (red) and 2 (blue).}\label{fig:sample}
	\includegraphics[scale = 1]{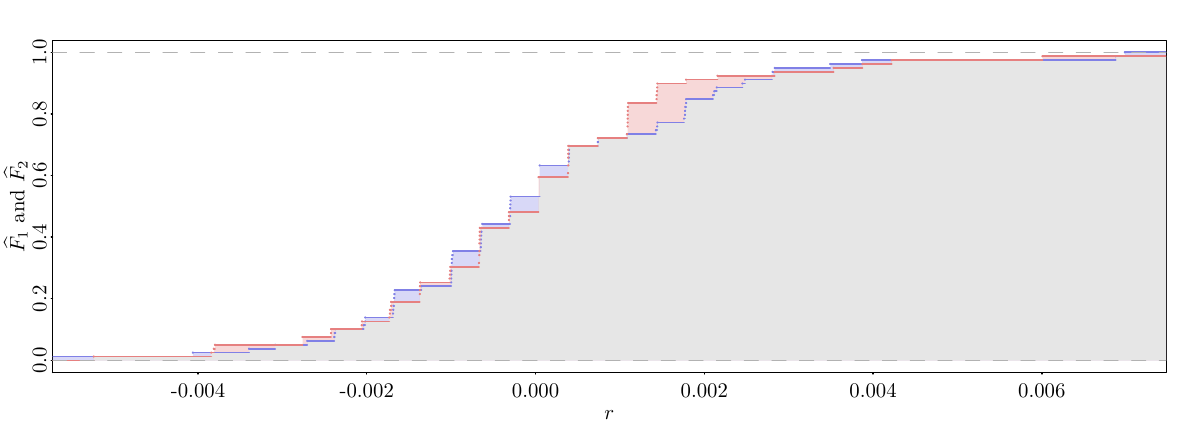}
\end{figure}

\begin{figure}%[H]
	\centering
	\caption{(a) 1st to 10th largest estimated eigenvalues; (b) log scale.}\label{fig:thetahat}
	\includegraphics[scale = 1]{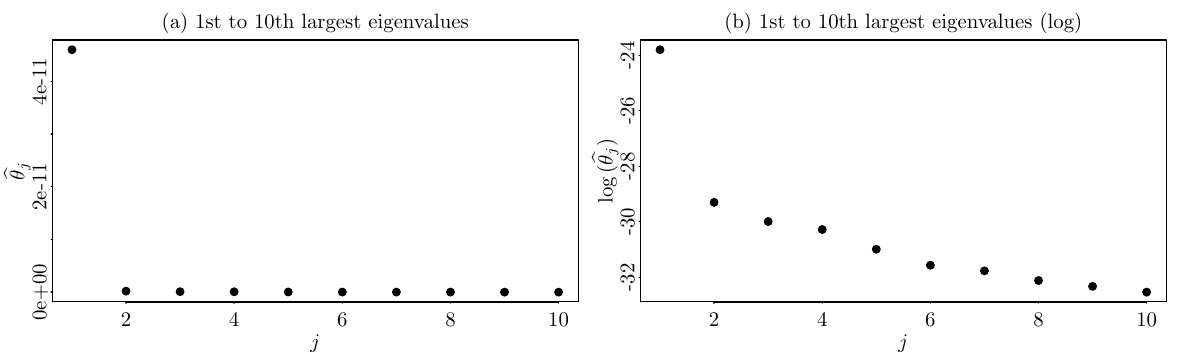}
\end{figure}

\begin{figure}%[H]
	\centering
	\caption{Estimated eigenfunction $\widehat{\psi}$.}\label{fig:psihat}
	\includegraphics[scale = 1]{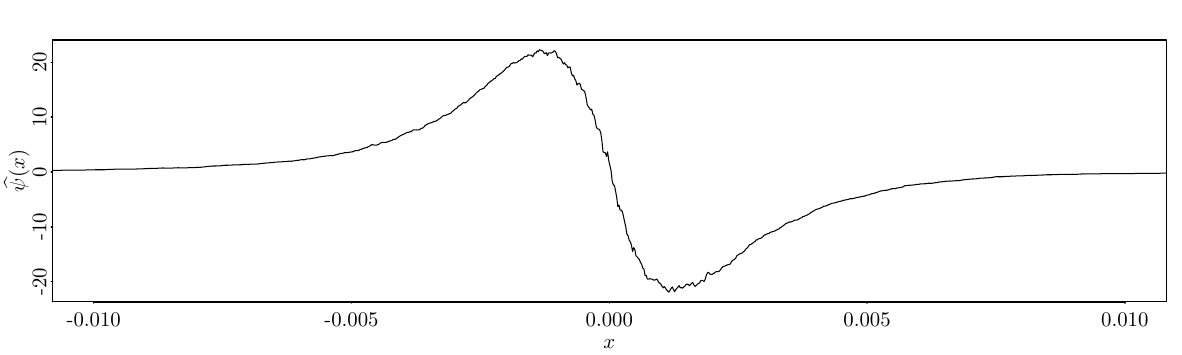}
\end{figure}

Figure~\ref{fig:thetahat} displays the largest estimated eigenvalues $\widehat{\theta}_j$ of $\widehat{R}^\mu$ (panel (b) plots them on a log scale). The drop from the largest to the second largest eigenvalue is markedly steep, whereas from the second to the third largest and so on the drop is much smoother. Moreover, the $p$--value from the Ljung-Box test for independence is nearly zero for the time series $\big(\widehat{W}_{t1}\big)$, $t=1,2,\dots,n$, whereas for $\big(\widehat{W}_{t2}\big)$ it is $0.04528$. This may be indicative that indeed there is dynamic dependence in the direction of $\psi_1$ but not in the remaining ones. Observe though that interpretation of $p$--values must be taken with caution in the present context, as pointed for instance in Bathia et al.~\cite[remark 3]{bathia2010identifying}. The sample path of the estimated $\widehat{W}_{1t}$ is found in Figure~\ref{fig:etahat-ts}, panel (a). The plot of the estimated eigenfunction $\widehat{\psi}_1$ is shown in Figure~\ref{fig:psihat}. It displays a plausible shape whereas the eigenfunction $\widehat{\psi}_2$ is very irregular (the plot is not reported here). In any case we assume $d=1$ and write $\widehat{W}_t\equiv \widehat{W}_{t1}$, and likewise $\widehat{\psi} \equiv \widehat{\psi}_1$. We then perform the augmented Dickey-Fuller test to the time series $\widehat{W}_{t}$, and the obtained $p$-values are virtually zero whatever specification is used, be it with a drift component, a drift and a trend component, or neither. Therefore we take $\widehat{W}_{t}$ to be stationary. Figure~\ref{fig:etahat-acf-pacf} displays the \textsc{acf} and \textsc{pacf} plots for $\widehat{W}_{t}$.

\begin{figure}%[H]
	\centering
	\caption{Estimated coefficients $\widehat{W}_t$: (a) time series plot; (b) lagged scatterplot.}\label{fig:etahat-ts}
	\includegraphics[scale = 1]{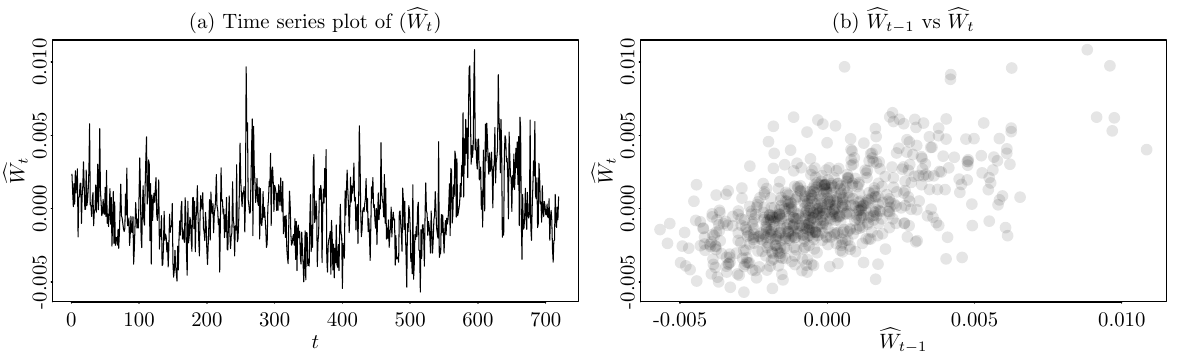}
\end{figure}

\begin{figure}%[H]
	\centering
	\caption{Correlation functions of $\widehat{W}_t$. (a) \textsc{acf}; (b) \textsc{pacf}.}\label{fig:etahat-acf-pacf}
	\includegraphics[scale = 1]{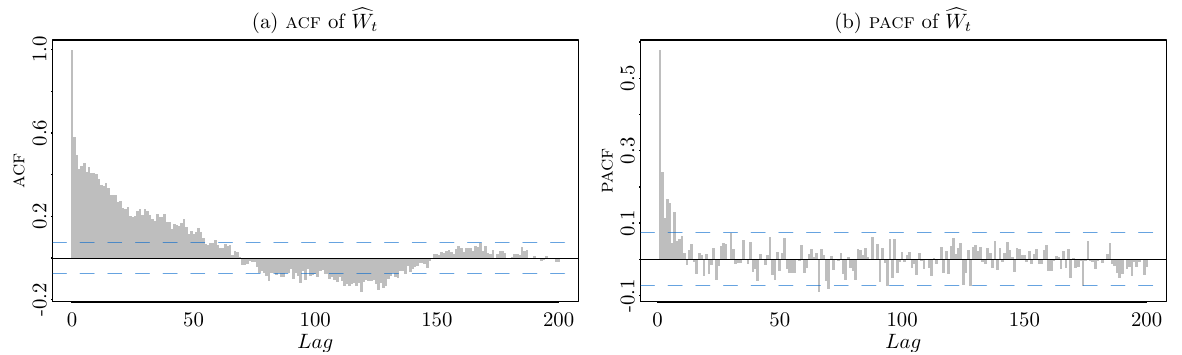}
\end{figure}

We now model the estimated latent time series $\big(\widehat{W}_t\big)$. Figure~\ref{fig:etahat-ts}, panel (b) displays the scatterplot of $\widehat{W}_{t-1}$ \textsc{vs} $\widehat{W}_t$, indicating a linear relationship. Together with the \textsc{acf} and \textsc{pacf} plots from Figure~\ref{fig:etahat-acf-pacf}, as well as the results from the augmented Dickey-Fuller tests discussed above, we feel authorized to assume that $\big(\widehat{W}_t\big)$ is an \textsc{arma} process. We choose the \textsc{arma}(1,2) specification based on the \textsc{aic} criterium. The estimation results can be found in Table~\ref{tab:eta-arma}.

\begin{table}%[H]
\centering{}\caption{Coefficient estimates and standard errors obtained from an \textsc{arma}(1,2) fit to the data $\left(\widehat{W}_1,\dots,\widehat{W}_n\right)$.}\label{tab:eta-arma}
\begin{centering}
\begin{tabular*}{10cm}{@{\extracolsep{\fill}}lcccc}
\toprule
  & \textsc{ar1}  & \textsc{ma1} & \textsc{ma2} & intercept \tabularnewline
\midrule
\midrule
Coef. & 0.9706 & -0.6470 & -0.1276 & 0e+00 \tabularnewline

s.e.  & 0.0116 &  0.0378 &  0.0353 & 5e-04 \tabularnewline
\midrule
\textsc{aic} = -6964.83 \tabularnewline
\bottomrule
\end{tabular*}
\par\end{centering}
\end{table}

By modeling and forecasting the time series $\big(\widehat{W}_t\big)$ it is possible to obtain forecasts of the upcoming latent distributions $F_{n+1}$. In this direction, it will be convenient -- although abusing a little on notation -- to redefine the full sample size as $N = 719$, and to write $\widehat{\E}_{|n}F_0$, $\widehat{\psi}_{|n}$, $\widehat{W}_{t|n}$, and $\widetilde{F}_{t|n}$ to denote the estimates of the corresponding quantities obtained through applying our methodology to a subsample of size $n\leq N$ (notice though that $\widehat{F}_t$ does not depend on $n$). Thus for instance $\widehat{\E}_{|N}F_0 = \widehat{\E}F_0$ in the previous notation. Now, letting $\widehat{W}_{n+1|n}$ denote the one--step--ahead forecast implied by an \textsc{arma}(1,2) fit to the data $\big(\widehat{W}_{1|n},\dots,\widehat{W}_{n|n}\big)$, we can define the forecast
\begin{equation*}
	\widetilde{F}_{n+1|n} := \widehat{\E}_{|n}F_0 + \widehat{W}_{n+1|n}\,\widehat{\psi}_{|n}.
\end{equation*}
Next, we set $n_0 = 350$ and generate forecasts $\widetilde{F}_{n+1|n}$ with $n$ ranging from $n_0$ to $N-1$. Figure~\ref{fig:Ftilde-norms} displays two such forecasts, corresponding to the indices $n\in\left\{n_0,\dots,N-1\right\}$ which minimize (resp. maximize) the norm $\Vert\widetilde{F}_{n+1|n} - \widehat{F}_{n+1}\Vert$. Notice that the object being forecast (namely, $F_{n+1}$) is not observable, not even \emph{ex post}, whence we use $\widehat{F}_{n+1}$ as a proxy.

\begin{figure}%[H]
	\centering
	\caption{\textsc{Cdf} forecast $\widetilde{F}_{n+1|n}$ (black) and \emph{ex post} empirical \textsc{cdf} $\widehat{F}_{n+1}$ (blue): (a) $n = \arg\min_{n_0\leq t<N} \Vert\widetilde{F}_{t+1|t} - \widehat{F}_{t+1}\Vert$; (b) $n = \arg\max_{n_0\leq t<N} \Vert\widetilde{F}_{t+1|t} - \widehat{F}_{t+1}\Vert$.}\label{fig:Ftilde-norms}
	\includegraphics[scale = 1]{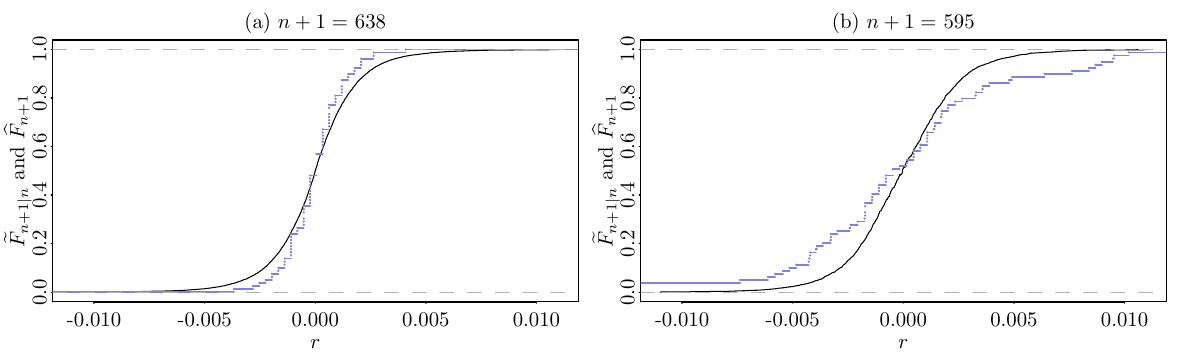}
\end{figure}
\bigskip

\paragraph{Modeling the conditional variance} Of particular interest in the present framework is the possibility of forecasting the conditional variance $\sigma_{n+1}^2:=\Var\left(X_{1,n+1} \vert \Xi\right)$, which describes the variability of the process $\left(X_\tau:\, \tau\in\left[n+1,n+2\right)\right)$ during cycle $n+1$. One such forecast is easily obtained from the above \textsc{cdf} forecasting procedure, by letting
\begin{equation*}
	\widetilde{\sigma}_{n+1|n}^2 := \int x^2\,\widetilde{F}_{n+1|n}\left(dx\right) - \left( \int x\,\widetilde{F}_{n+1|n}\left(dx\right) \right)^2.
\end{equation*}
Let us call this approach the \emph{forecasting strategy 1}.

Alternatively, one may consider a model analogous to the benchmark \textsc{har} Realized Volatility model (see Corsi~\cite{Corsi19022009} and Corsi et al.~\cite{corsi2012har} for a review), given by
\begin{equation}\label{eq:HAR-RV}
	\log\left(\widehat{\sigma}_{t+1}^2\right) = \beta_0 + \beta_1\log\left(\widehat{\sigma}_t^2\right) + \beta_2\log\left(\widehat{\sigma}_{t,\textsc{w}}^2\right) + \texttt{error},
\end{equation}
where $\widehat{\sigma}_t^2$ is the empirical variance of the sample $X_{1,t},\dots,X_{q_t,t}$, and where $\widehat{\sigma}_{t,\textsc{w}}^2 := \left(\widehat{\sigma}_t^2 + \cdots + \widehat{\sigma}_{t-4}^2\right)/5$ is a heterogeneous weekly component (we drop the monthly component for convenience). Here the error term is assumed to satisfy some regularity conditions such as being serially uncorrelated, having zero conditional mean (or median), etc. The above model specification is legitimate in the present context, since the intra--cycle sample variances are proportional -- under a mild assumption of uncorrelated returns -- to the daily realized volatilities, which are described by an equation similar to \eqref{eq:HAR-RV}. We thus let $\widehat{\sigma}_{n+1|n}^2$, $n_0\leq n \leq N-1$, denote the exponential of the one--step--ahead forecasts obtained from a median regression fit of model \eqref{eq:HAR-RV} to the data $\widehat{\sigma}_1^2,\dots,\widehat{\sigma}_n^2$. Let us call this approach the \emph{forecasting strategy 2}.

Notice that here, as it was the case in the context of forecasting the latent \textsc{cdf}s, again the `true' quantity being forecast (namely, $\sigma_{t+1}^2$) is not observable, not even \emph{ex post}. Thus our predictions are contrasted with empirical realizations, which are taken as proxies for their population counterparts; for instance the mean squared error of forecasting strategy 1 above is calculated as
\begin{equation*}
	\frac{1}{N-n_0}\sum_{n=n_0}^{N-1} \left( \widetilde{\sigma}_{n+1|n}^2 - \widehat{\sigma}_{n+1}^2 \right)^2.
\end{equation*}
Figure~\ref{fig:sigmahat-fcst-boxplot} displays the boxplots of the forecast errors $\widetilde{\sigma}_{n+1|n}^2 - \widehat{\sigma}_{n+1}^2$ and $\widehat{\sigma}_{n+1|n}^2 - \widehat{\sigma}_{n+1}^2$, across the forecasting horizon (that is, with $n$ ranging from $n_0$ to $N-1$). Heuristically, one would expect that the forecasting strategies which use the $\widehat{W}_t$ in their formulation would display better forecasting power since each of the $\widehat{W}_t$ is constructed using full sample information, whereas $\widehat{\sigma}_t^2$ only uses information from day $t$. This reasoning is supported by the relative mean squared error of strategy 1 with respect to strategy 2, which is approximately $0.91$ in this data set. Furthermore, at the $5\%$ level we reject the null hypothesis that strategies 1 and 2 have equal forecasting accuracy, in favor of the alternative that strategy 1 has greater forecasting accuracy (with the Diebold--Mariano test statistic having a $p$--value equal to $0.03718$). Aside from these improvements, we call attention to the fact that, by forecasting the latent $F_{n+1}$, one can readily obtain a derived forecast of any quantity that is attached to that \textsc{cdf}, such as quantiles (especially useful in the context of Value-at-Risk evaluation), probabilities of certain events of interest, skewness and kurtosis of the distribution, etc.
% n_0 + 0, n_0 + 1, n_0 + 2, ..., n_0 +(n-n_0) - 1
% n - 1= n_0 +(n-n_0) -1

\begin{figure}%[H]
	\centering
	\caption{Boxplot of the statistics $\widetilde{\sigma}_{n+1|n}^2 - \widehat{\sigma}_{n+1}^2$ and $\widehat{\sigma}_{n+1|n}^2 - \widehat{\sigma}_{n+1}^2$, with $n_0\leq n < N$: (a) no outliers are plotted; (b) outliers are plotted.}\label{fig:sigmahat-fcst-boxplot}
	\includegraphics[scale = 1]{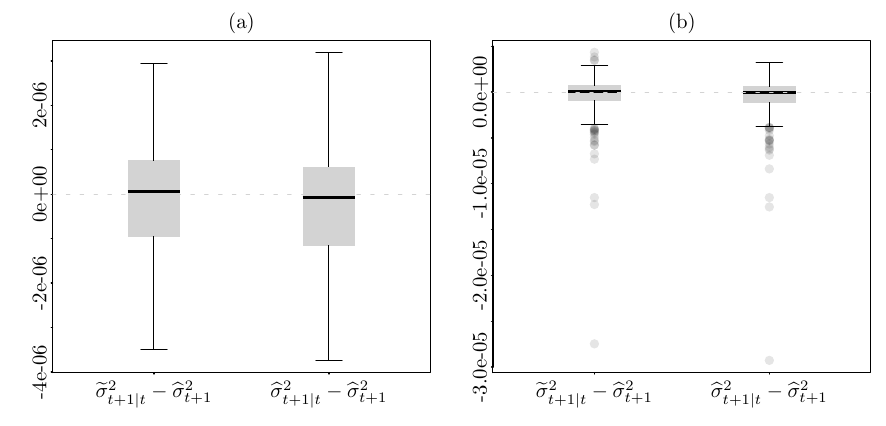}
\end{figure}

\bigskip

A last comment on how to interpret the obtained forecasts may come in handy. At the end of day $t$, the statistician can apply our methodology and obtain, say, a forecast $\widehat{\sigma}_{t+1\vert t}^2$ for the variance $\sigma_{t+1}^2$. The latter quantity is the variance of a 5--minute return at any instant during day $t+1$, as implied by the common marginal distribution of these returns via equation~\eqref{eq:weakly-conjugate-model}. Thus someone who negotiates in the market in 5--minute intervals has `typical' variance equal to $\sigma_{t+1}^2$; this is the quantity that we are forecasting, and thus $\widehat{\sigma}_{t+1\vert t}^2$ estimates the `typical' variability someone who negotiates at each 5 minutes would expect to find next day. It is convenient to mention that a forecast obtained through a high frequency \textsc{garch} fit to the data available up to the end of day $t$ would have a different interpretation and thus would not be comparable to our method. Indeed, the two modeling frameworks may be difficult to connect. For example, at the end of day $t$, the \textsc{garch} model will give a forecast for the variance of the `opening' return rather than for the typical variance of $5$--minute returns over day~$t+1$.

Lastly, recall that in the framework of conjugate processes, a question of its own interest is identification of the dimension $d$ and characterization of the dynamics of $\big(\widehat{\boldsymbol{W}}_t\big)$. In this regard,  we can say that there is some evidence in the data that the true dimension is indeed equal to one, and that the latent $\left(W_t\right)$ sequence is a linear \textsc{arma} process. Testing these and other hypotheses would require derivation of the asymptotic distribution of our estimators and is beyond the scope of the present paper. Advancements in this regard can be found in Mas~\cite{mas2002weak} and Kokoszka and Reimherr~\cite{kokoszka2013asymptotic}.

\newpage

\begin{appendices}

\numberwithin{equation}{section}

\section{A bit of theory}\label{sec:bit-theory}
Let $M_1\left(\mathbb{R}\right)$ denote the space of all Borel probability measures on $\mathbb{R}$. In this paper we always consider $M_1\left(\mathbb{R}\right)$ endowed with the weak* topology, that is, the topology in which a sequence $\mu_n$ of probability measures converges to a probability measure $\mu$ if and only if $\int f\,d\mu_n\rightarrow \int f\,d\mu$ for each continuous bounded function $f$ on $\mathbb{R}$. This topology is metrizable by the Lévy--Prohorov metric
\begin{equation*}
	\rho\left(\mu,\nu\right) := \inf\left\{\delta\geq 0:\,\mu\left(B\right)<\nu\left(B^\delta\right)+\delta\quad\mbox{and}\quad\nu\left(B\right)<\mu\left(B^\delta\right)+\delta,\,\forall B\right\},
\end{equation*}
where $B^\delta$ is the $\delta$--neighborhood of any set $B\subset\mathbb{R}$.

Let $\xi$ be a random element in $M_1\left(\mathbb{R}\right)$. Obviously, if $f$ is a continuous bounded function on $\mathbb{R}$, then $\int f\,d\xi$ is a real random variable. Recall that the \emph{baricenter of $\xi$} is the unique element $\E\xi\in M_1\left(\mathbb{R}\right)$ such that the equality $\E\int f\,d\xi = \int f\,d\E\xi$ holds for all continuous bounded $f\colon\mathbb{R}\rightarrow\mathbb{R}$. The baricenter $\E\xi$ is also known as the \emph{Pettis integral} of $\xi$ with respect to $\Prob$, or as the \emph{expectation} of $\xi$.

\begin{lemma}\label{thm:random-measure-basics}
Let $\xi$ be a random element in $M_1\left(\mathbb{R}\right)$ and let $\E\xi$ be its baricenter. Then
\begin{enumerate}[label={(\textit{\roman*}})]
	\item For each Borel set $B$, $\xi\left(B\right)$ is a random variable;
	\item $\E\xi\left(B\right) = \E\left(\xi\left(B\right)\right)$; \label{thm:random-measure-basics-item-ii}
	\item $\supp\xi\subset\supp\E\xi$ almost surely.
\end{enumerate}
\end{lemma}
\begin{proof}
For the first item, let $f_B\left(\mu\right):=\mu\left(B\right)$ be defined for $\mu\in M_1\left(\mathbb{R}\right)$. It is sufficient to show that $f_B$ is measurable, since $\xi\left(B\right) = f_B\circ\xi$. If $K$ is closed, then $f_K$ is lower semicontinuous, by the Portmanteau Theorem, and thus measurable. By a similar argument $f_U$ is measurable if $U$ is open. For the general $B$, there are some open sets $\left\{U_j\right\}$ such that $\mu\left(B\right) = \mu\left(\bigcap U_j\right)$, and thus $f_B\left(\mu\right) = \mu\left(B\right) = \mu\left(\bigcap U_j\right) = \lim \mu\left(U_j\right) = \lim f_{U_j}\left(\mu\right)$ which establishes measurability of $f_B$.

For the third item, let $U:=\mathbb{R}\setminus\supp\left(\E\xi\right)$. Then $\xi\left(U\right)\geq 0$ and $\E\left(\xi\left(U\right)\right) = \E\xi\left(U\right) = 0$, by item~\ref{thm:random-measure-basics-item-ii}. Hence $\xi\left(U\right) = 0$ almost surely.

The second assertion is left as an exercise.
\end{proof}

\begin{lemma}\label{thm:hilbert-embedding}
Let $\xi$ be a random element in $M_1\left(\mathbb{R}\right)$, and let $\E\xi$ be its baricenter . Define $F$ by
\begin{equation*}
F\left(x\right) := \xi\left(-\infty,x\right],\qquad x\in\mathbb{R}.
\end{equation*}
If $\mu$ is a finite measure on $\mathbb{R}$, absolutely continuous with respect to Lebesgue measure on an interval containing $\supp\left(\E\xi\right)$, then $F$ is a bounded, strongly measurable random element in $L^2\left(\mu\right)$. Moreover, the Bochner expectation of $F$ is the map $x\mapsto \E\xi\left(-\infty, x\right]$.
\end{lemma}

\begin{proof}
For $\nu\in M_1\left(\mathbb{R}\right)$, let $\bar{\nu}\colon\mathbb{R}\rightarrow\mathbb{R}$ be defined by $\bar{\nu}\left(x\right):=\nu\left(-\infty,x\right]$. Clearly $\bar{\nu}$ is measurable and bounded, and hence one has $\bar{\nu}\in L^2\left(\mu\right)$, for each $\nu \in M_1\left(\mathbb{R}\right)$.
It is sufficient to prove that the map $\nu\mapsto \bar{\nu}$ from $M_1\left(\mathbb{R}\right)$ to $L^2\left(\mu\right)$ is continuous. Let $\nu_n\rightarrow\nu$ in $M_1\left(\mathbb{R}\right)$. By the Portmanteau Theorem, $\bar{\nu}_n\left(x\right)\rightarrow \bar{\nu}\left(x\right)$ for each $x$ at which $\bar{\nu}$ is continuous. The set of discontinuity points of $\bar{\nu}$ is at most countable and is contained in $\supp\left(\E\xi\right)$, and hence has $\mu$ measure $0$. That is, $\abs{\bar{\nu}_n\left(x\right) - \bar{\nu}\left(x\right)}^2\rightarrow 0$ for $\mu$--almost all $x$. Moreover, $\abs{\bar{\nu}_n\left(x\right) - \bar{\nu}\left(x\right)}^2 \leq 1$ and hence the Lebesgue Dominated Convergence Theorem gives $\int \abs{\bar{\nu}_n\left(x\right) - \bar{\nu}\left(x\right)}^2\,\mu\left(dx\right)\rightarrow 0$. This establishes continuity of $\nu\mapsto\bar{\nu}$. The remaining assertions are left as an exercise.
\end{proof}

\begin{theorem}\label{thm:xi-orthogonal-ker}
Let $\mathscr{H}$ be a separable Hilbert space and $H$ be a centered random element in $\mathscr{H}$ of strong second order, with covariance operator $C$. Then $H \perp \Null(C)$ almost surely.
\end{theorem}

\begin{proof}
See \cite[Theorem~1]{horta2015identifying}.
\end{proof}

\begin{corollary}\label{thm:hilbert-representations}
In the conditions of Theorem~\ref{thm:xi-orthogonal-ker}, let $d:=\rank\left(C\right)$, where possibly $d=\infty$. Let $\left(\lambda_j\right)_{j=1}^d$ be the sequence of nonzero eigenvalues of $C$ (with repetitions), and $\left(\varphi_j\right)_{j=1}^d$ be the associated sequence of orthonormal eigenvectors. Then
\begin{equation*}
H(\omega) = \sum_{j=1}^d\langle H(\omega),\varphi_j\rangle\varphi_j
\end{equation*}
in $\mathscr{H}$ almost surely. Moreover, the scalar random variables $\langle H,\varphi_j\rangle$ are such that $\E\langle H,\varphi_j\rangle = 0$, $\E\langle H,\varphi_j\rangle^2 = \lambda_j$, and $\E \langle H,\varphi_i\rangle \langle H,\varphi_j\rangle = 0$ if $i\neq j$.
\end{corollary}
\begin{proof}[Proof of Corollary~\ref{thm:hilbert-representations}]
	This is left as an exercise.
\end{proof}

\section{Proofs}\label{sec:proofs}
In this section we shall write $\Xi := \sigma\left(\xi_0,\xi_1,\dots\right)$. Whenever we need to indicate dependence of $\xi_t$ or $F_t$ on the sample space, we shall write $\xi_t^\omega$, $F_t^\omega$, etc. For the random variables $X_\tau$ we write $X_\tau\left(\omega\right)$ as usual.

\begin{proof}[Proof of Lemma~\ref{thm:X_tau-unconditional-distribution}]
One has $\Prob\left(X_\tau\in B\right) = \E\left[\Prob\left(X_\tau\in B\,\vert\, \xi_0,\xi_1,\dots\right)\right] = \E\left[\xi_{\floor{\tau}}\left(B\right)\right] = \E\xi_0\left(B\right)$ by Lemma~\ref{thm:random-measure-basics} and Assumption~S.
\end{proof}

\begin{proof}[Proof of Theorem~\ref{thm:WLLN}.]
Without loss of generality, let $\tau_i = i$. Write $\xi_t\left(f\right) = \int f\,d\xi_t$ and similarly $\E\xi_0\left(f\right) = \int f\,d\E\xi_0$. Notice that $\E\xi_0\left(f\right) = \E f\circ X_\tau$, by Lemma~\ref{thm:X_tau-unconditional-distribution}. Let $Y_t = f\circ X_t - \xi_t\left(f\right)$. We have
\begin{equation*}
\abs{\frac{1}{n}\sum_{t=0}^{n-1} f\circ X_t - \E\xi_0\left(f\right)} \leq \abs{\frac{1}{n}\sum_{t=0}^{n-1} Y_t} + \abs{\frac{1}{n}\sum_{t=0}^{n-1} \xi_t\left(f\right) - \E\xi_0\left(f\right)}.
\end{equation*}
The second term in the above sum is $o_\Prob\left(1\right)$ by hypothesis. For the first term, given $\epsilon>0$ we have
\begin{equation}\label{eq:WLLN-2}
\Prob\left\{ \abs{\frac{1}{n}\sum_{t=0}^{n-1} Y_t} > \epsilon \right\} = \E\left[\Prob\left\{ \abs{\frac{1}{n}\sum_{t=0}^{n-1} Y_t} > \epsilon \,\bigg\vert\,\Xi\right\}\right].
\end{equation}
But $\left(Y_t\vert\Xi:\,t=0,1,\dots\right)$ is an independent sequence, with $\E\left[Y_t\vert\Xi\right] = 0$, and therefore a conditional \textsc{lln} gives
\begin{equation*}
	\lim_{n\rightarrow \infty}\Prob\left\{ \abs{\frac{1}{n}\sum_{t=0}^{n-1} Y_t} > \epsilon \,\bigg\vert\,\Xi\right\} = 0,\qquad\mbox{almost surely}.
\end{equation*}
Now the sequence in the above equation is bounded by $1$ and hence the RHS in \eqref{eq:WLLN-2} goes to zero by the Lebesgue Dominated Convergence Theorem.
\end{proof}

\begin{proof}[Proof of Lemma~\ref{thm:hilbert-embedding-F_t}.]
The fact that each $F_t$ is a random element in $L^2\left(\mu\right)$ is a direct consequence of Lemma~\ref{thm:hilbert-embedding}. The remaining assertions are easily established.
\end{proof}

\begin{proof}[Proof of Theorem~\ref{thm:spectral-representation-F_t}]
The first three assertions, and equation~\eqref{eq:F_t-spectral-representation-C_0}, follow directly from Theorem~\ref{thm:xi-orthogonal-ker} and Corollary~\ref{thm:hilbert-representations}. It remains to show that each $\varphi_j$ is a bounded, càdlàg function. To establish boundedness, notice that since $\sup_{\left(x,y\right)}\abs{C_0\left(x,y\right)}\leq 2$, one has
\begin{align*}
	\abs{\lambda_j\varphi_j\left(x\right)} &= \abs{C_0^\mu\varphi_j\left(x\right)} \\
	&\leq 2\int \abs{\varphi_j\left(y\right)}\,\mu\left(dy\right).
\end{align*}
Thus, as $\varphi_j\in L^1\left(\mu\right)$, one has $\sup_x\abs{\varphi_j\left(x\right)}<\infty$.

Now let $x_n\uparrow x\in\mathbb{R}$. Put $h_n\left(\omega\right) := F_0^\omega\left(x_n\right)\int F_0^\omega\left(y\right)\varphi_j\left(y\right)\,\mu\left(dy\right)$, and define $h\left(\omega\right)=\lim h_n\left(\omega\right)$. Now $\abs{h_n}$ is bounded by a constant and hence the Lebesgue Dominated Convergence Theorem gives $\lim\E h_n = \E h$. Since $\E h_n = \lambda_j\varphi_j\left(x_n\right)$, the làg part of the assertion is proved. For the càd part a similar argument can be followed. The details are left as an exercise.
\end{proof}

\begin{proof}[Proof of Proposition~\ref{thm:mu-nu-equivalence}]
The fact that $\rank\left(C_0^\mu\right) = \rank\left(C_0^\nu\right)$ is a direct consequence of the argument below. By definition $\mu$ and $\nu$ are equivalent to Lebesgue measure on the interval $I_0:=$ `intersection of all closed intervals containing $\supp\left(\E\xi\right)$'. If $I_0$ is bounded and $x\notin I_0$, then it is easily seen that $C_0\left(x,y\right) = 0$ and hence $\varphi_i^\mu\left(x\right) = \varphi_i^\nu\left(x\right) = 0$. In any case we can assume without loss of generality that $\supp\left(\mu\right) = \supp\left(\nu\right) = I_0$.
Let
\begin{equation*}
	g_n\left(x,y\right) := C_0\left(x,y\right) - \sum_{j=1}^n \lambda_j^\mu\varphi_j^\mu\left(x\right)\varphi_j^\mu\left(y\right)
\end{equation*}
By Theorem~3.a.1 in K\"{o}nig~\cite{konig2013eigenvalue}, there is a measurable subset $E\subset\mathbb{R}\times\mathbb{R}$ such that $\mu\otimes\mu\left(E^c\right) = 0$ and $L_n := \sup_{\left(x,y\right)\in E} \abs{g_n\left(x,y\right)}\rightarrow 0$. Now since $\nu\ll\mu$, we have $\nu\otimes\mu\ll\mu\otimes\mu$ and thus $\nu\otimes\mu\left(E^c\right) = 0$. Thus, the Minkowski integral inequality (Lieb and Loss~\cite[Theorem~2.4]{lieb2001analysis}) gives
\begin{align*}
	\norm{\lambda_i^\nu\varphi_i^\nu - \sum_{j=1}^n\lambda_j^\mu\varphi_j^\mu \langle \varphi_j^\mu, \varphi_i^\nu\rangle_\nu}_\mu &= \left\{\int \left( \int {g_n\left(x,y\right) \varphi_i^\nu\left(y\right)} \nu\left(dy\right) \right)^2 \mu\left(dx\right)\right\}^{1/2} \\
	&\leq \int \left( \int g_n\left(x,y\right)^2\varphi_i^\nu\left(y\right)^2 \mu\left(dx\right) \right)^{1/2} \nu\left(dy\right) \\
	&= \int \left( \int g_n\left(x,y\right)^2\varphi_i^\nu\left(y\right)^2\,\mathbb{I}_E\left(x,y\right) \mu\left(dx\right) \right)^{1/2} \nu\left(dy\right) \\
	&\leq L_n \int \left( \int \varphi_i^\nu\left(y\right)^2\,\mathbb{I}_E\left(x,y\right) \mu\left(dx\right) \right)^{1/2} \nu\left(dy\right) \rightarrow 0.
\end{align*}
This shows that $\left\{\varphi_1^\nu, \varphi_2^\nu,\dots\right\}$ is contained in the $L^2\left(\mu\right)$--closure of $\left\{\varphi_1^\mu, \varphi_2^\mu,\dots\right\}$. A similar computation establishes that
\begin{equation*}
	\norm{\lambda_i^\mu\varphi_i^\mu - \sum_{j=1}^n\lambda_j^\nu\varphi_j^\nu \langle \varphi_j^\nu, \varphi_i^\mu\rangle_\mu}_\mu \rightarrow 0,
\end{equation*}
that is, $\left\{\varphi_1^\mu, \varphi_2^\mu,\dots\right\}$ is contained in the $L^2\left(\mu\right)$--closure of $\left\{\varphi_1^\nu, \varphi_2^\nu,\dots\right\}$. This concludes the proof.
\end{proof}

\bigskip

\begin{proof}[Proof of Lemma~\ref{thm:epsilon-is-well-behaved}]
For item \ref{thm:epsilon-is-well-behaved-item1}, we have
\begin{equation*}
	\E\big[\varepsilon_t\left(x\right)\,\vert\,\Xi\big] = \E\big[\widehat{F}_t\left(x\right) - F_t\left(x\right)\vert \Xi\big] = \E\big[\widehat{F}_t\left(x\right)\vert\Xi\big] - F_t\left(x\right).
\end{equation*}
Now
\begin{equation}\label{eq:EG_t-useful-identity}
	\E\big[\widehat{F}_t\left(x\right)\vert\Xi\big] = \frac{1}{q_t}\sum_{i=1}^{q_t}\E\left[\mathbb{I}_{\left\{X_{it}\leq x\right\}}\vert\Xi\right] =F_t\left(x\right).
\end{equation}
This yields the stated equality. For \ref{thm:epsilon-is-well-behaved-item2} it is immediate that we have, by item~\ref{thm:epsilon-is-well-behaved-item1}, the equality $\E\left[F_t\left(x\right)\varepsilon_{t+k}\left(y\right)\,\vert\,\Xi\right] = F_t\left(x\right)\E\left[\varepsilon_{t+k}\left(y\right)\vert\Xi\right]=0$.

To establish \ref{thm:epsilon-is-well-behaved-item3}, a straightforward calculation and using \eqref{eq:EG_t-useful-identity} will give
\begin{align*}
\E\big[\varepsilon_t\left(x\right)\varepsilon_{t+k}\left(y\right)\,\vert\,\Xi\big] 
% &= \E\left[\left(\widehat{F}_t\left(x\right)-F_t\left(x\right)\right)\left(\widehat{F}_{t+k}\left(y\right)-F_{t+k}\left(y\right)\right)\right] \\
% &= \E\left\{ \E\left[\widehat{F}_t\left(x\right)\widehat{F}_{t+k}\left(y\right)\vert\Xi\right] - F_{t+k}\left(y\right)\E\left[\widehat{F}_t\left(x\right)\vert\Xi\right] \right\} \\
% & \quad\quad\quad + \E\left\{F_t\left(x\right)F_{t+k}\left(y\right) - F_t\left(x\right)\E\left[\widehat{F}_{t+k}\left(y\right)\vert\Xi\right]\right\} \\
&= \E\big[\widehat{F}_t\left(x\right)\widehat{F}_{t+k}\left(y\right)\,\vert\,\Xi\big] - F_{t+k}\left(y\right)F_t\left(x\right).
\end{align*}
Then
\begin{align*}
\E\left[\widehat{F}_t\left(x\right)\widehat{F}_{t+k}\left(y\right)\,\vert\,\Xi\right]=\frac{1}{q_t q_{t+k}}\sum_{i=1}^{q_t}\sum_{j=1}^{q_{t+k}}\E\left[\mathbb{I}_{\left[X_{it}\leq x\right]}\mathbb{I}_{\left[X_{j,t+k}\leq y\right]}\,\vert\,\Xi\right],
\end{align*}
but
\begin{align*}
\E\left[\mathbb{I}_{\left[X_{it}\leq x\right]}\mathbb{I}_{\left[X_{j,t+k}\leq y\right]} \vert \Xi\right] &= \Prob\left[X_{it}\leq x,\,X_{j,t+k}\leq y \vert \Xi\right] = F_t\left(x\right)F_{t+k}\left(y\right)
\end{align*}
by the cyclic--independence assumption via \eqref{eq:cyclic-independence}. This yields the stated result.
\end{proof}

\begin{proof}[Proof of Theorem~\ref{thm:spectral-representation-F_t-Rmu}]
	This is an immediate consequence of the stated assumptions and is left as an exercise.
\end{proof}

\begin{proof}[Proof of Lemma~\ref{thm:finite-rank-equivalence1}]
Without loss of generality, let $k=1$. We have
\begin{equation*}
C_1\left(x,y\right)=\sum_{i=1}^d\sum_{j=1}^d\left(\E Z_{0i}Z_{1j}\right)\varphi_i\left(x\right)\varphi_j\left(y\right).
\end{equation*}
This establishes $\Ran\left(C_1^\mu\right)\subset \Ran\left(C_0^\mu\right)$. Clearly the reverse inclusion holds if and only if $\left(\E Z_{0i}Z_{1j}\right)_{ij}$ is full-rank, as stated. The details can be found in the proof of Proposition~1 in Bathia et al.~\cite{bathia2010identifying}.
\end{proof}

\begin{proof}[Proof of Corollary~\ref{thm:finite-rank-equivalence2}]
For simplicity and without loss of generality, assume that the required full--rank property holds with $k=1$. Let $p = 1$. Thus $R^\mu$ is the integral operator with kernel $$R_\mu\left(x,y\right)=\int C_1\left(x,z\right)C_1\left(y,z\right)\,\mu\left(dz\right).$$ Put another way we have $R^\mu = C_1^\mu\left(C_1^\mu\right)^*$, where $C_1^\mu$ is the integral operator with kernel $C_1\left(x,y\right)$ and $*$ means adjoining. Now use Lemma~\ref{thm:finite-rank-equivalence1} and the fact that $\overline{\Ran\left(TT^*\right)} = \overline{\Ran\left(T\right)}$ for any bounded linear operator $T$ on a separable Hilbert space.
\end{proof}

\begin{proof}[Proof of Theorem~\ref{thm:LLN-for-Fhat}]
Recall that the map $\nu\mapsto\left(x\mapsto \nu\left(-\infty,x\right]\right)$ from $M_1\left(\mathbb{R}\right)$ to $L^2\left(\mu\right)$ is continuous (see the proof of Lemma~\ref{thm:hilbert-embedding}). Thus, by the Continuous Mapping Theorem, the assumption that $n^{-1}\sum_{t=1}^n\xi_t\rightarrow\E\xi_0$ in probability implies that $n^{-1}\sum_{t=1}^n F_t\rightarrow\E F_0$ in probability (in $L^2\left(\mu\right)$). This establishes \ref{thm:LLN-for-Fhat-item-i}. Item \ref{thm:LLN-for-Fhat-item-ii} is left as an exercise.

For the last assertion in the Theorem, let $	H_{t} := \widehat{F}_t - F_t$. Observe that $H_{t}$ is a strong order $2$ (indeed, bounded) random element in the Hilbert space $L^2\left(\mu\right)$. Now notice that
\begin{equation*}
	\Big\Vert \widehat{\E}F_0 - \E{F_0}\Big\Vert \leq \norm{\frac{1}{n}\sum_{t=1}^n H_{t}} + \norm{\frac{1}{n}\sum_{t=1}^n F_t - \E{F_0}}.
\end{equation*}
The second term in the above sum is $O_\Prob\left(n^{-1/2}\right)$ by assumption. For the first term, we will need the following result.
\begin{lemma}[Hilbert space Hoeffding Inequality. {Boucheron et al.~\cite[p. 172]{boucheron2013concentration}}]
Let $H_1,\dots,H_n$ be independent, centered random elements in a separable Hilbert space $\mathscr{H}$. If for some $c>0$ one has $\norm{H_t}\leq c/2$ for all $t$, then for each $\epsilon \geq c/2$ it holds that
\begin{equation*}
	\Prob\left[\norm{\sum_{t=1}^n H_t} > \sqrt{n}\epsilon\right] \leq \exp\left(-\frac{\left(\epsilon - c/2\right)^2}{c^2/2} \right).
\end{equation*}
\end{lemma}

\bigskip
The assumption that $\left(\xi_t,X_\tau\right)$ is cyclic--independent ensures that $(H_{t}\vert\Xi:\,t=1,2,\dots)$ is an independent sequence of centered random elements in $L^2\left(\mu\right)$, with $\norm{H_{t}}\leq \sqrt{\abs{\mu}}$. Thus, for $c = 2\sqrt{\abs{\mu}}$, we have
\begin{equation*}
\Prob\left[\Big\Vert \sum_{t=1}^n H_{t} \Big\Vert > \sqrt{n}\epsilon\, \Bigg\vert \Xi \right] \leq \exp\left(-\frac{\left(\epsilon - c/2\right)^2}{c^2/2} \right),\qquad\mbox{almost surely.}
\end{equation*}
Taking expectation on both sides yields the stated result.
\end{proof}

\begin{proof}[Proof of Theorem~\ref{thm:LLN-for-psihat-no-rate}]
	Without loss of generality, let $p = 1$. Recall that
	\begin{equation*}
		C_1\left(x,y\right) = \E F_0\left(x\right)F_1\left(y\right) - \E F_0\left(x\right)\E F_0\left(y\right).
	\end{equation*}
	A straightforward calculation establishes that
	\begin{equation*}
		\widehat{C}_1\left(x,y\right) = \frac{1}{n-1}\sum_{t=1}^{n-1}\widehat{F}_t\left(x\right)\widehat{F}_{t+1}\left(y\right) \quad - \quad \widehat{\E}F_0\left(x\right)\widehat{E}F_0\left(y\right) + o_\Prob\left(1\right)
	\end{equation*}
	in $L^2\left(\mu\otimes\mu\right)$. Since, by Theorem~\ref{thm:LLN-for-Fhat}, we have that $\widehat{\E}F_0\left(x\right)\widehat{E}F_0\left(y\right) = \E F_0\left(x\right)\E F_0\left(y\right) + o_\Prob\left(1\right)$ in $L^2\left(\mu\otimes\mu\right)$, it only remains to show that
	\begin{equation*}
		\frac{1}{n}\sum_{t=1}^{n}\widehat{F}_t\left(x\right)\widehat{F}_{t+1}\left(y\right) = \E F_0\left(x\right)F_1\left(y\right) + o_\Prob\left(1\right)
	\end{equation*}
	in $L^2\left(\mu\otimes\mu\right)$. We have
	\begin{align*}
		{\frac{1}{n}\sum_{t=1}^n \widehat{F}_t\left(x\right)\widehat{F}_{t+1}\left(y\right) - \E F_0\left(x\right)\E F_0\left(y\right)} &=
		\left(\frac{1}{n}\sum_{t=1}^n \widehat{F}_t\left(x\right)\widehat{F}_{t+1}\left(y\right) - \frac{1}{n}\sum_{t=1}^n {F}_t\left(x\right){F}_{t+1}\left(y\right)\right) \\
		&\quad + \left(\frac{1}{n}\sum_{t=1}^n {F}_t\left(x\right){F}_{t+1}\left(y\right) - \E{F}_0\left(x\right)\E{F}_0\left(y\right)\right).
	\end{align*}
	The second term in the above sum is $o_\Prob\left(1\right)$ in $L^2\left(\mu\otimes\mu\right)$, by hypothesis. For the first term, let
	\begin{equation*}
		H_t\left(x,y\right):= \widehat{F}_t\left(x\right)\widehat{F}_{t+1}\left(y\right) - F_t\left(x\right)F_{t+1}\left(y\right).
	\end{equation*}
	 Clearly $\big(H_t:\,t=1,2,\dots\big)$ is a sequence of random elements in $L^2\left(\mu\otimes\mu\right)$. Moreover, the assumption that $\left(\xi_t,X_\tau\right)$ is cyclic--independent ensures that $\big(H_{2t}\vert\Xi:\,t=1,2,\dots\big)$ is a \emph{centered}, \emph{independent} sequence of random elements in $L^2\left(\mu\otimes\mu\right)$. The same is true of $\big(H_{2t-1}\vert\Xi:\,t=1,2,\dots\big)$. Thus, by the Law of Large Numbers for Hilbert spaces, one has for each $\epsilon>0$
	 \begin{equation*}
	 	\lim_{n\rightarrow\infty}\Prob\left(\norm{\frac{1}{n}\sum_{t=1}^n H_{2t}}_{L^2\left(\mu\otimes\mu\right)} > \epsilon\,\bigg\vert\, \Xi\right) = 0,\qquad \mbox{almost surely},
	 \end{equation*}
	and since the above sequence is bounded by $1$, the Lebesgue Dominated Convergence Theorem gives
	\begin{equation*}
	 	\lim_{n\rightarrow\infty}\Prob\left(\norm{\frac{1}{n}\sum_{t=1}^n H_{2t}}_{L^2\left(\mu\otimes\mu\right)} > \epsilon\right) \rightarrow 0.
	 \end{equation*}
	 This establishes that $\widehat{C}_1\left(x,y\right) = C_1\left(x,y\right) + o_\Prob\left(1\right)$ in $L^2\left(\mu\otimes\mu\right)$, and hence $\Vert\widehat{C}_1^\mu - C_1^\mu\Vert_{HS} = o_\Prob\left(1\right)$. Continuity of the operations of adjoining and composition imply that $\Vert\widehat{R}^\mu - R^\mu\Vert = o_\Prob\left(1\right)$. The remaining assertions are an immediate consequence of Theorem~1.1 in Mas and Menneteau~\cite{mas2003perturbation}.
\end{proof}

\begin{proof}[Proof of Theorem~\ref{thm:LLN-for-psihat}]
Notice that condition C2 in Bathia et al.~\cite{bathia2010identifying} is always satisfied in our setting. Their conditions C1 and C3 correspond to the assumptions in Theorem~\ref{thm:LLN-for-psihat}. Condition C4 there is item~\ref{thm:epsilon-is-well-behaved-item2} in our Lemma~\ref{thm:epsilon-is-well-behaved}. It only remains to observe that their proof is valid in any separable Hilbert space and not only in $L^2\left(\left[a,b\right]\right)$.
\end{proof}

\begin{proof}[Proof of Proposition~\ref{thm:LLN_etahat}]
Notice that both $\psi_j$ and $-\psi_j$ are normalized eigenfunctions of $R^\mu$. We assume that the `right' one has been picked. To obtain \eqref{eq:eta-hat-asymptotics}, add and subtract the terms $\langle\widehat{F}_t,\psi_j\rangle$ and $\langle\E F_0,\psi_j\rangle$ to \eqref{eq:W-hat-definition},  and then use \eqref{eq:empirical-df-equals-F-plus-error} and \eqref{eq:W_tj-definition}; $\widehat{\rho}_j$ is defined implicitly in this procedure. The bound on $\abs{\widehat{\rho}_j}$ is just an application of the Cauchy--Schwarz inequality, together with the fact that $\norm{\psi_j}$ and $\Vert\widehat{\psi}_j\Vert$ are equal to 1 by construction, and that both $\Vert\widehat{\E}F_0\Vert$ and $\Vert \widehat{F}_t\Vert$ are bounded by the square root of $\abs{\mu}:=\mu\left(\mathbb{R}\right)$. The remaining assertions are an immediate consequence of Lemma~\ref{thm:epsilon-is-well-behaved}.
\end{proof}

\begin{proof}[Proof of Corollary~\ref{thm:LLN2_etahat}]
The Corollary is an immediate consequence of the stated assumptions. The details are left as an exercise.
\end{proof}

\begin{proof}[Proof of Proposition~\ref{thm:d-hat-consistency}]
	See Theorem~3 in Bathia et al.~\cite{bathia2010identifying} and the Appendix~B therein.
\end{proof}

\begin{proof}[Proof of Corollary~\ref{thm:F-tilde-asymptotics}]
	The corollary is a direct consequence of the stated hypotheses and is left as an exercise for the reader.
\end{proof}

\section{Estimation procedure and numerical computation}\label{sec:estimation}
This section describes how one can obtain estimates of the $\psi_j$ and $W_{tj}$ through straightforward matrix analysis. We shall restrict our attention to the scenario where $d<\infty$. This approach is adopted by Bathia et al.~\cite{bathia2010identifying}. The idea is to represent the operator $\widehat{R}^\mu$ as an infinite matrix acting on the canonical Hilbert space $\ell^2$, and then to obtain a $\left(n-p\right)\times\left(n-p\right)$ matrix whose spectrum coincides with that of $\widehat{R}^\mu$. The construction relies on the fact that given any operators $A$ and $B$, it is always true that $AB^*$ and $B^*A$ share the same nonzero eigenvalues. The representation of $\widehat{R}^\mu$ is given by the $\infty\times\infty$ matrix
\begin{equation*}
	\frac{1}{\left(n-p\right)^{2}}\boldsymbol{G}_{0}\sum_{k=1}^{p}\boldsymbol{G}_{k}'\boldsymbol{G}_{k}\boldsymbol{G}_{0}',
\end{equation*}
where $\boldsymbol{G}_{k}=\left[\boldsymbol{g}_{1+k}\ \dots\ \boldsymbol{g}_{n-p+k}\right]$ and $\boldsymbol{g}_t\in\ell^2$ is such that $\boldsymbol{g}_t'\boldsymbol{g}_s = \langle \widehat{F}_t-\widehat{\E}F_0, \widehat{F}_s - \widehat{\E}F_0\rangle$. Now apply the duality discussed above with $A = \boldsymbol{G}_0$ and \linebreak
%LINEBREAK?
$B = \sum_{k=1}^{p}\boldsymbol{G}_k'\boldsymbol{G}_k\boldsymbol{G}_0'$ to obtain the $\left(n-p\right)\times\left(n-p\right)$ matrix
\begin{equation}
\boldsymbol{M}:=\frac{1}{\left(n-p\right)^{2}}\sum_{k=1}^{p}\boldsymbol{G}_{k}'\boldsymbol{G}_{k}\boldsymbol{G}_{0}'\boldsymbol{G}_{0}.\label{eq:K_star}
\end{equation}
To be explicit, the entry $\left(t,s\right)$ of $\boldsymbol{G}_k'\boldsymbol{G}_k$ is the inner product $\big\langle \widehat{F}_{t+k} - \widehat{\E}F_0,\, \widehat{F}_{s+k} - \widehat{\E}F_0\,\big\rangle $. The preceding heuristics establishes the first claim of the following Proposition.

\begin{proposition}\label{thm:matrix-Rmu-equivalence}
\label{thm:BYZ_prop2} The $\left(n-p\right)\times\left(n-p\right)$ matrix $\boldsymbol{M}$ shares the same nonzero eigenvalues with the operator $\widehat{R}^\mu$. Moreover, the associated eigenfunctions of $\widehat{R}^\mu$ are given by
\begin{equation}
\widetilde{\psi}_{j}\left(x\right)=\sum_{t=1}^{n-p}\gamma_{jt}\big(\widehat{F}_{t}\left(x\right)-\widehat{\E}F_0\left(x\right)\big),\label{eq:psi_tilde}
\end{equation}
 where $\gamma_{jt}$ is the $t$-th component of the eigenvector $\boldsymbol{\gamma}_{j}$ associated to the $j$-th largest eigenvalue of $\boldsymbol{M}$.
\end{proposition}
\begin{proof} See Proposition~2 in Bathia et al.~\cite{bathia2010identifying} and the Appendix~B therein.
\end{proof}

We then let $\widehat{\psi}_{j}:=\widetilde{\psi}_j / \Vert \widetilde{\psi}_j \Vert $ denote the normalized eigenfunctions of $\widehat{R}^\mu$. Notice that in order to obtain the matrix $\boldsymbol{M}$ all one needs is to calculate the inner products $\langle \widehat{F}_t - \widehat{\E}F_0,\,\widehat{F}_s-\widehat{\E}F_0\rangle$ with $t$ and $s$ ranging from $1$ to $n$. An important aspect in our context is that, unlike general Functional Data Analysis methodologies, the explicit formulas for this coefficients can be easily derived. Indeed,
\begin{equation*}
	\big\langle \widehat{F}_t - \widehat{\E}F_0,\,\widehat{F}_s-\widehat{\E}F_0\big\rangle = \langle \widehat{F}_t, \widehat{F}_s\rangle - \langle \widehat{F}_t, \widehat{\E}F_0\rangle - \langle \widehat{F}_s, \widehat{\E}F_0\rangle + \langle \widehat{\E}F_0, \widehat{\E}F_0\rangle,
\end{equation*}
with
\begin{align*}
	\langle \widehat{F}_t, \widehat{F}_s\rangle &= \frac{1}{q_t q_s}\sum_{i=1}^{q_t}\sum_{j=1}^{q_s}\mu\big[ X_{it}\vee X_{js},\,+\infty\big), \\
	\langle\widehat{\E}F_0,\widehat{\E}F_0\rangle &= \frac{1}{n^2}\sum_{t=1}^{n}\sum_{s=1}^{n}\langle \widehat{F}_t, \widehat{F}_s\rangle, \\
	\langle \widehat{F}_t, \widehat{\E}F_0\rangle &= \frac{1}{n}\sum_{s=1}^{n}\langle \widehat{F}_t, \widehat{F}_s\rangle.
\end{align*}
The norms $\Vert\widetilde{\psi}_j\Vert$ can be calculated as well through
\begin{equation*}
	\Vert\widetilde{\psi}_j\Vert^2 = \sum_{t=1}^{n-p}\sum_{s=1}^{n-p}\gamma_{jt}\gamma_{js}\langle \widehat{F}_t - \widehat{\E}F_0, \widehat{F}_s - \widehat{\E}F_0\rangle,
\end{equation*}
and finally the coefficients $\widehat{W}_{tj}$ are given by
\begin{equation*}
	\widehat{W}_{tj} = \frac{1}{\Vert\widetilde{\psi}_j\Vert}\sum_{s=1}^{n-p}\gamma_{js}\langle \widehat{F}_t - \widehat{\E}F_0, \widehat{F}_s - \widehat{\E}F_0\rangle.
\end{equation*}
Computationally, the above formulas are advantageous (in comparison with calculating the inner--products via numerical integration) because they are exact; however, the expression for $\langle \widehat{F}_t, \widehat{F}_s\rangle$ indicates that at least \emph{some} computational cost is inescapable.

\end{appendices}
\section*{References}

%\bibliography{references}

\end{document}